\documentclass[11pt]{amsart}

\title[Divides with cusps, shadows, and transvergent diagrams]
{Divides with cusps, shadows, and transvergent diagrams}

\author{Ryoga Furutani}
\address{
Department of Mathematics~ \slash ~International Institute for Sustainability with Knotted Chiral Meta Matter (WPI-SKCM$^2$), 
Hiroshima University, 1-3-1 Kagamiyama, Higashi-Hiroshima, Hiroshima 739-8526, Japan}
\email{ryoga.furutani0409@gmail.com}

\usepackage{amsthm}
\usepackage{mathrsfs}
\usepackage{latexsym}
\usepackage[dvipdfmx]{graphicx}
\usepackage[dvips]{psfrag}
\usepackage[dvips]{color}
\usepackage{xypic}
\usepackage[abs]{overpic}
\usepackage{caption}
\usepackage[all]{xy}
\usepackage[normalem]{ulem}

\usepackage{amscd}

\theoremstyle{plain}
\newtheorem*{theorem*}{Theorem}
\newtheorem*{lemma*} {Lemma}
\newtheorem*{corollary*} {Corollary}
\newtheorem*{proposition*}{Proposition}
\newtheorem*{conjecture*}{Conjecture}
\newtheorem{theorem}{Theorem}[section]

\newtheorem{proposition}[theorem]{Proposition}

\theoremstyle{remark}

\newtheorem*{definition}{Definition}
\newtheorem*{claim*}{Claim}

\newtheorem{example}{Example}

\theoremstyle{definition}

\newtheoremstyle{citing}
  {}
  {}
  {\itshape}
  {}
  {\bfseries}
  {.}
  {.5em}
  {\thmnote{#3}}

\theoremstyle{citing}

\newcommand{\ZZ}{\mathbb{Z}}

\newcommand{\CC}{\mathbb{C}}

\newcommand{\Nbd}{\operatorname{Nbd}}

\newcommand{\Int}{\operatorname{Int}}

\newcommand\erase{\bgroup\markoverwith{\textcolor{red}{\rule[.5ex]{2pt}{0.4pt}}}\ULon}



\begin{document}

\maketitle

\begin{abstract}
A link $L$ in $S^3$ is called a symmetric link if it is preserved by a $\pi$ rotation around a closed geodesic in $S^3$. 
Any symmetric link can be depicted by a diagram with a symmetry axis lying on the plane of the diagram, called a transvergent diagram. 
Recently, Sugawara proved that any symmetric link can be represented by a divide with cusps, which is a generalization of A'Campo's divide that allows a finite number of cusps. 
In this paper, we introduce a generalization of A'Campo's divide in terms of Turaev's shadow, called a divide with gleams. 
By using divides with gleams, we provide an algorithm to obtain a divide with cusps that represents a symmetric link from its given transvergent diagram. Conversely, we also provide an algorithm to draw a transvergent diagram of the link of a given divide with cusps. 
\end{abstract}

\vspace{1em}

\begin{small}
\hspace{2em}  \textbf{2020 Mathematics Subject Classification}: 
57K10; 57Q60, 57R05


\hspace{2em} 
\end{small}

\section*{Introduction}
A link $L$ in $S^{3}$ is called a {\it symmetric link} if there exists a $\pi$ rotation $\iota$ around a closed geodesic in $S^3$ such that $\iota(L) = L$. A {\it diagram} of a link $L$ is the image of a generic projection of $L$ from $S^{3}$ to $S^{2}$ equipped with over/under information at each crossing of the image.
A link diagram is called a {\it transvergent diagram} if there exists a reflection through a great circle on $S^{2}$ that preserves the diagram 
(see Figure \ref{figure:transvergent_diagram}).
\begin{figure}[htbp]
\centering\includegraphics[width=6cm]{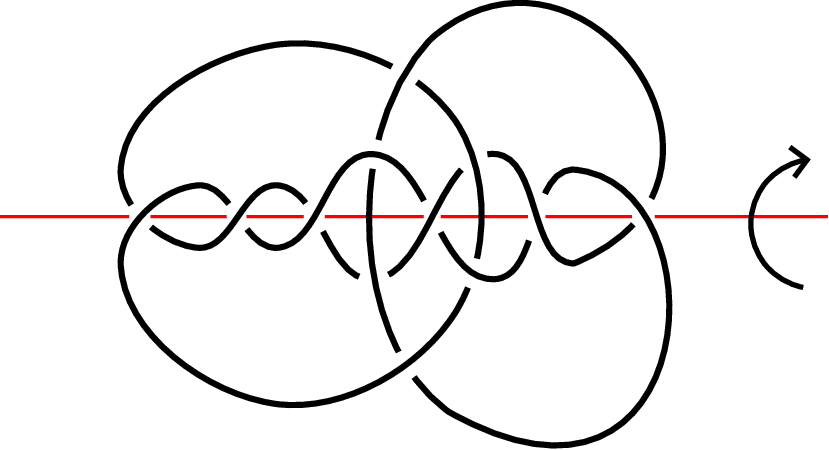}
\caption{A transvergent diagram of a symmetric link.}
\label{figure:transvergent_diagram}
\end{figure}
It is easily seen that any link with a transvergent diagram is a symmetric link, and conversely, any symmetric link has a transvergent diagram. 
A transvergent diagram is used as a fundamental tool in the study of strongly invertible knots or $2$-periodic knots (see e.g. \cite{BI22, HHS23, L22, L23}).

A {\it divide} is the image of a proper and generic immersion of a compact $1$-manifold into the $2$-disk. A divide is a concept introduced by A'Campo \cite{Camp98} in his study on isolated singularities of complex plane curves. 
Each divide $P$ defines a link $L_{P}$ in $S^{3}$, called the {\it link of the divide} $P$. 
A {\it divide with cusps} is a divide with a finite number of cusps, which is a generalization of A'Campo's divide introduced by Sugawara-Yoshinaga  \cite{SY22} in their study of the complements of the complexifications of real line arrangements.
Each divide with cusps $P$ also defines a link $L_{P}$ in $S^{3}$, called the {\it link of the divide with cusps} $P$. 
Sugawara \cite{S24} proved that a link $L$ in $S^{3}$ is a symmetric link if and only if there exists a divide with cusps $P$ whose associated link $L_P$ is equivalent to $L$.

In this paper, we provide a transformation between these two representations of a symmetric link, a divide with cusps and a transvergent diagram. 
To state the main theorem of this paper, we introduce several notations and terminologies. 
Let $L$ be a symmetric link in $S^{3}$. 
Let $DL$ be the transvergent diagram of $L$. 
By forgetting the over/under information at the crossings of $DL$, 
we obtain an immersed curve on $S^{2}$. 
By the definition of a transvergent diagram, 
there exists a reflection of $S^{2}$ through a great circle of $S^{2}$ that preserves this immersed curve. 
We then take the quotient of $S^2$ and the immersed curve by this reflection, obtaining a $2$–disk $D$ together with an immersed curve in $D$.
In general, the resulting immersed curve intersects $\partial D$ at the parts corresponding to the crossing points of the original immersed curve that lie on the great circle of $S^2$.
We perturb the immersed curve slightly so as to remove these boundary intersections.
The resulting immersed curve in the disk $D$ is a divide, and we denote it by $P_{DL}$
(see Figure \ref{figure:DL_to_P_DL}). 
The divide $P_{DL}$ is called the {\it divide of the transvergent diagram $DL$}. 
Then, the following theorem holds:
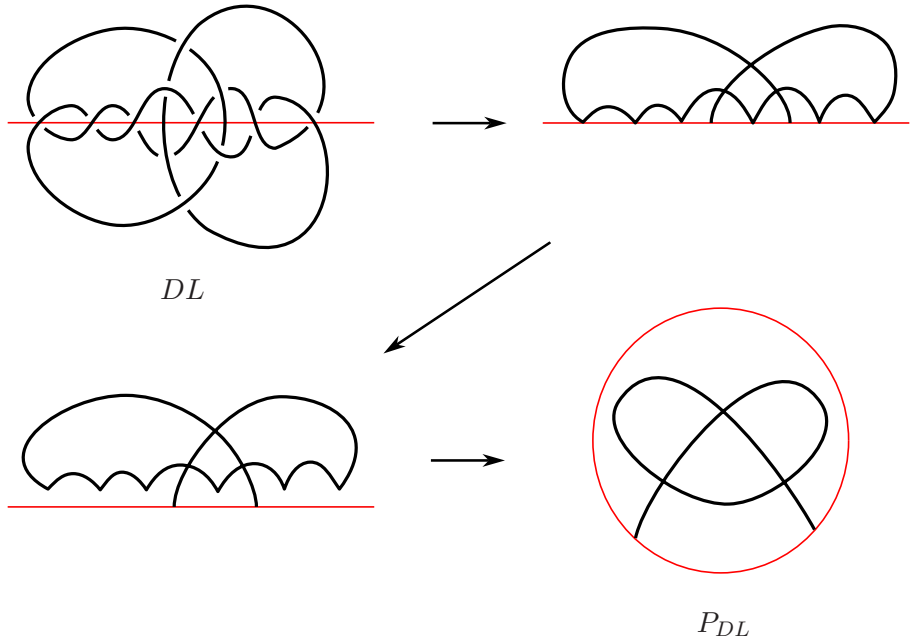
\begin{figure}[htbp]
\centering
\includegraphics[width=12cm]{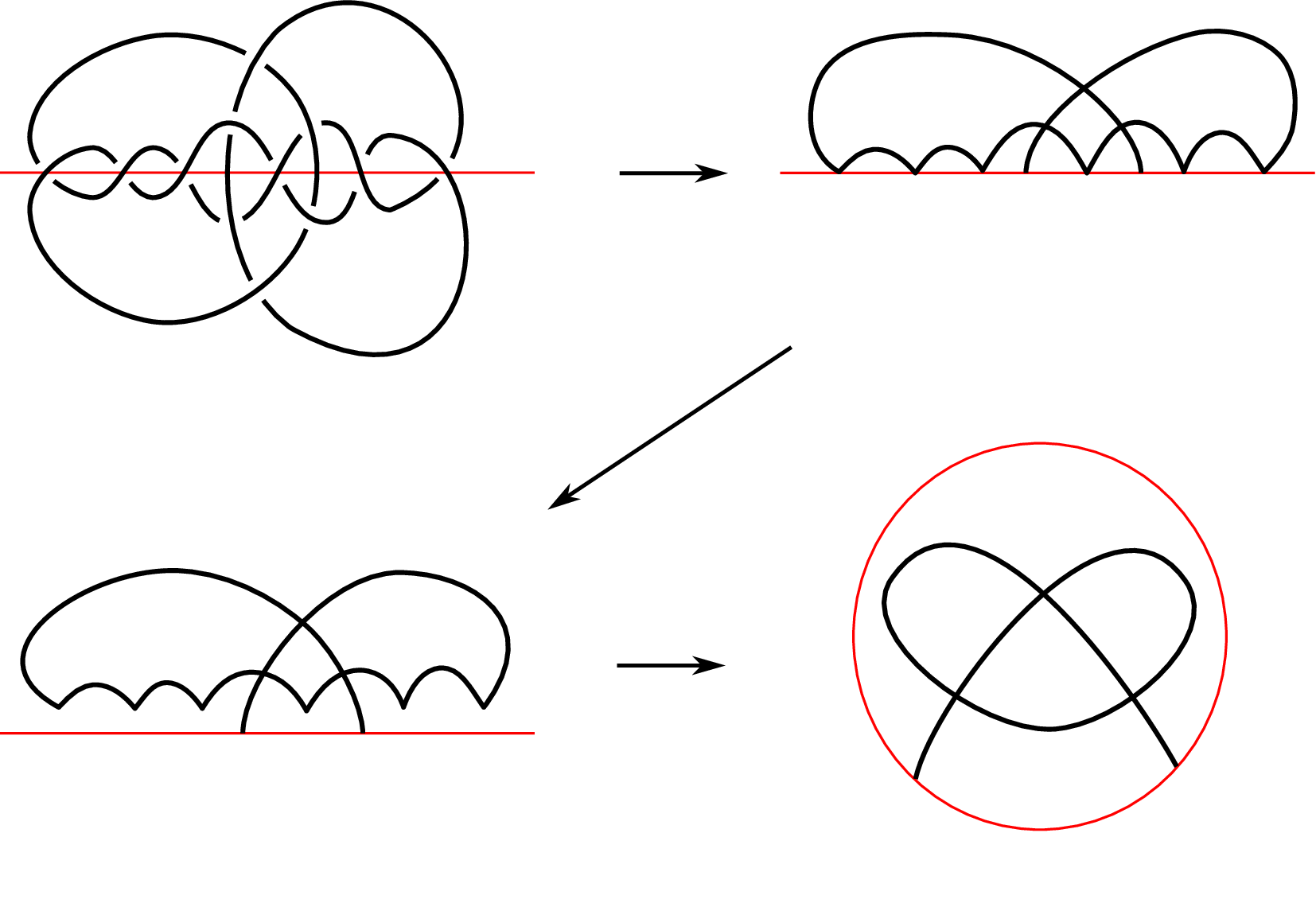}
\caption{The transvergent diagram $DL$ defines a divide $P_{DL}$.}
\label{figure:DL_to_P_DL}
\begin{picture}(400,0)(0,0)
\put(87,158){$DL$}
\put(290,32){$P_{DL}$}
\end{picture}
\end{figure} 

\begin{theorem}
\label{theo:main}
Let $L$ be a symmetric link, $DL$ be its transvergent diagram, and $P_{DL}$ be the divide of $DL$. Then, there exists a divide with cusps $P$ satisfying the following conditions:
\begin{itemize}
\item $P$ is obtained by adding a finite number of cusps to $P_{DL}$.
\item The link $L_P$ of $P$ is equivalent to the link $L$.
\end{itemize}
\end{theorem}

We will give a constructive proof of Theorem \ref{theo:main}. That is, we provide an algorithm for constructing a divide with cusps that satisfies the conditions of Theorem \ref{theo:main} from a given transvergent diagram. 
The main tool for the proof of Theorem \ref{theo:main} is Turaev's shadow introduced by Turaev \cite {Tur94}.
A shadow is a $2$-dimensional simple polyhedron that is properly, locally flatly embedded in a compact oriented $4$-manifold, such that the $4$-manifold collapses onto it.
A shadow is colored with a half-integer for each of its internal region, which is a generalization of the Euler number of a $D^2$-bundle over an orientable closed surface.
This coloring is called a {\it gleam}. 
A simple polyhedron $X$ equipped with a gleam $gl$ is called a {\it shadowed polyhedron}, and is denoted by $(X, gl)$. 
Turaev \cite {Tur94} gave a way to uniquely construct, up to diffeomorphism, a compact oriented $4$-manifold $W_X$ embedding the simple polyhedron $X$ as a shadow, from a given shadowed polyhedron $(X, gl)$. 
Moreover, since the embedding of the shadowed polyhedron $X$ into the $4$-manifold $W_X$ is proper, 
the boundary $\partial X$ of $X$ is embedded into $\partial W_X$, and thus forms a trivalent graph in $\partial W_X$.
Hence, a shadowed polyhedron simultaneously describes a compact oriented $4$-manifold and a trivalent graph in its boundary. 
Recently, a relationship between shadows and divides has been studied. 
Ishikawa-Naoe \cite{IN20} introduced a method for constructing a shadowed polyhedron from a divide on a compact orientable surface by using the doubling method, in order to describe the Lefschetz fibration associated with the divide. 
Here, a {\it divide on a compact orientable surface} is a generalization of A'Campo's divide on the $2$-disk, extended to a compact orientable surface, as introduced by Ishikawa \cite{I04}. 
Koda and the author \cite{FK23} established a new connection between the combinatorial structures of divides and the hyperbolic structures on the complements of links of divides by extending the work of Costantino-Thurston \cite{CT08} on the relationship between shadows and hyperbolic structures on $3$-manifolds. 
In this paper, we present a new method for constructing a shadowed polyhedron from a divide by using the natural involution of the tangent bundle of the disk. 
Through this method, we introduce a {\it divide with gleams}, which is a divide equipped with a coloring of its internal regions with half-integers. 
In Section \ref{section:proof_of_main_theorem}, we then show that a transvergent diagram $DL$ of a link $L$ naturally determines the divide with gleams $(P_{DL}, gl)$.
We then prove Theorem~\ref{theo:main} by presenting a constructive method that, starting from the gleam $gl$ of the divide with gleams $(P_{DL},gl)$, attaches cusps to $P_{DL}$ so that the resulting divide with cusps has the same associated link as $L$. 
Conversely, for a given divide with cusps, we can also construct a divide with gleams that represents its link. In Section \ref{section:A transvergent diagram of the link of a divide with cusps}, by using such a divide with gleams, we provide an algorithm to draw a transvergent diagram of the link of the given divide with cusps. 

\section{Preliminaries}
\label{section:preliminaries}

\subsection{Divides with cusps}
\label{subsection:divides_with_cusps}
Let $D := \{(x_{1}, x_{2}) \in \mathbb{R}^{2} \mid x_{1}^{2} + x_{2}^{2} \leq 1\}$ be the $2$-disk. 
\begin{definition}
Let $J$ be a compact $1$-manifold. 
A {\it divide with cusps} $P$ is the image of a smooth map $\alpha : J \to D$ satisfying the following conditions:

\begin{itemize}
\item $\alpha$ has finitely many cusp singularities $\{p_{1}, p_{2} \ldots, p_{n}\}$ in the interior of $J$, and the restriction of $\alpha$ to $J \setminus \{p_{1}, p_{2}, \ldots, p_{n}\}$ is an immersion with normal crossings and no triple points. Moreover, for any cusp singularity $p$, $\alpha(p)$ is not a double point of $P$.

\item $\alpha^{-1}(\partial D) = \partial J$, and $\alpha|_{\partial J}$ is injective. Moreover, $P$ intersects $\partial D$ transversely.
\end{itemize}
For a cusp singularity $p \in J$ of $\alpha$, the point $\alpha(p) \in P$ is called a {\it cusp} of $P$. 
A divide with cusps is called a {\it divide} if it does not have any cusps, which is precisely A'Campo's divide (A'Campo \cite{Camp98}). 
For a divide with cusps $P$, a point where $P$ intersects $\partial D$ is called an {\it endpoint} of $P$. 
A connected component of the set obtained from $P$ by removing all the double points and endpoints of $P$ is called an {\it edge} of $P$. 
A connected component of $D \setminus P$ is called a {\it region} of $P$. 
A region of $P$ that has no intersection with $\partial D$ is called an {\it internal region}. Otherwise, a region is called an {\it external region}. 
For a region $R$ of $P$, a cusp of $P$ directed toward the inside of $R$ is called an {\it inner cusp for} $R$. Otherwise, a cusp is called an {\it outer cusp for} $R$ (see Figure \ref{figure:divide_with_cusps}).
\begin{figure}[htbp]
\centering\includegraphics[width=6cm]{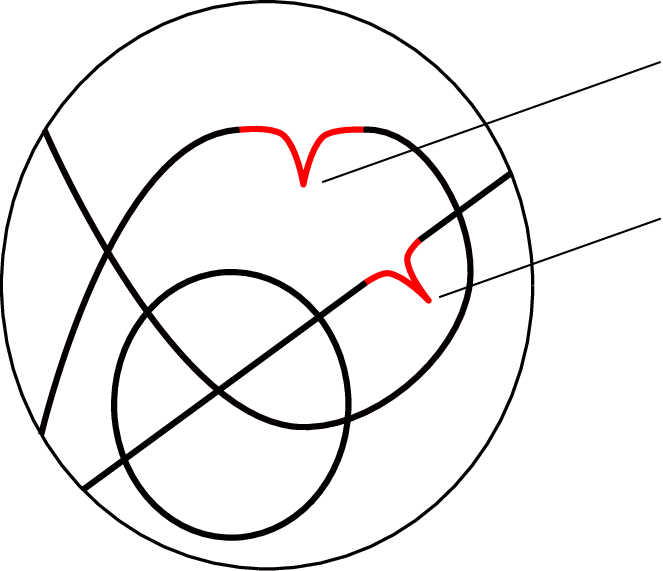}
\caption{A divide with cusps $P$. A region $R$ of $P$, an inner cusp for $R$ and an outer cusp for $R$.}
\begin{picture}(400,0)(0,0)
\put(190,130){$R$}
\put(285,181){an inner cusp for $R$}
\put(285,142){an outer cusp for $R$}
\end{picture}
\label{figure:divide_with_cusps}
\end{figure}
\end{definition} 
We identify the tangent bundle $TD$ of $D$ with $D \times \CC$, and identify the $4$-ball $D^{4}$ with the subset of the tangent bundle $TD$ defined as $\{(x, u) \in TD \mid |x|^{2} + |u|^{2} \leq 1\}$. We identify the $3$-sphere $S^3$ with $\partial D^4 = \{(x, u) \in TD \mid |x|^{2} + |u|^{2} = 1\}$.

\begin{definition}
Let $P$ be a divide with cusps. The {\it link $L_{P}$ of the divide with cusps $P$} is the link in $S^{3}$ defined by
$$L_{P} := \{(x, u) \in S^3 \mid x \in P, u \in T_x P \}.$$
If $P$ be a divide, then the link $L_{P}$ in $S^{3}$ is called the {\it link of the divide $P$}.
\end{definition}

Let $\iota : D^4 \to D^4$ be the involution of $D^4$ defined by $\iota(x,u) = (x, -u)$. 
The restriction $\iota |_{S^3} : S^3 \to S^3$ is the $\pi$ rotation of $S^3$ around the circle $\{(x,u) \in S^3 \mid u=0 \}$. 
By the definition of the link of a divide with cusps, the $\pi$ rotation $\iota |_{S^3}$ of $S^3$ preserves the link of a divide with cusps. 
Thus, the link of a divide with cusps is a symmetric link in $S^3$. 
As shown in the following theorem, Sugawara characterized symmetric links in terms of divides with cusps. 
\begin{theorem}[Sugawara \cite{S24}]
\label{theo:sugawara}
A link $L$ in $S^{3}$ is a symmetric link if and only if there is a divide with cusps $P$ such that $L_{P}$ is equivalent to $L$. Furthermore, if the link $L$ is strongly invertible, the divide with cusps $P$ consists of interval components and circle components, where the circle components have an even number of cusps. If $L$ is $2$-periodic, the divide with cusps $P$ consists only of circle components.
\end{theorem}

\begin{example}
Let $P$ be the divide with cusps shown in Figure \ref{figure:example_divide_with_cusps}.
Then, the link $L_{P}$ of $P$ is the figure-eight knot $4_1$ (see Sugawara \cite{S24}). 
\begin{figure}[htbp]
\centering\includegraphics[width=13cm]{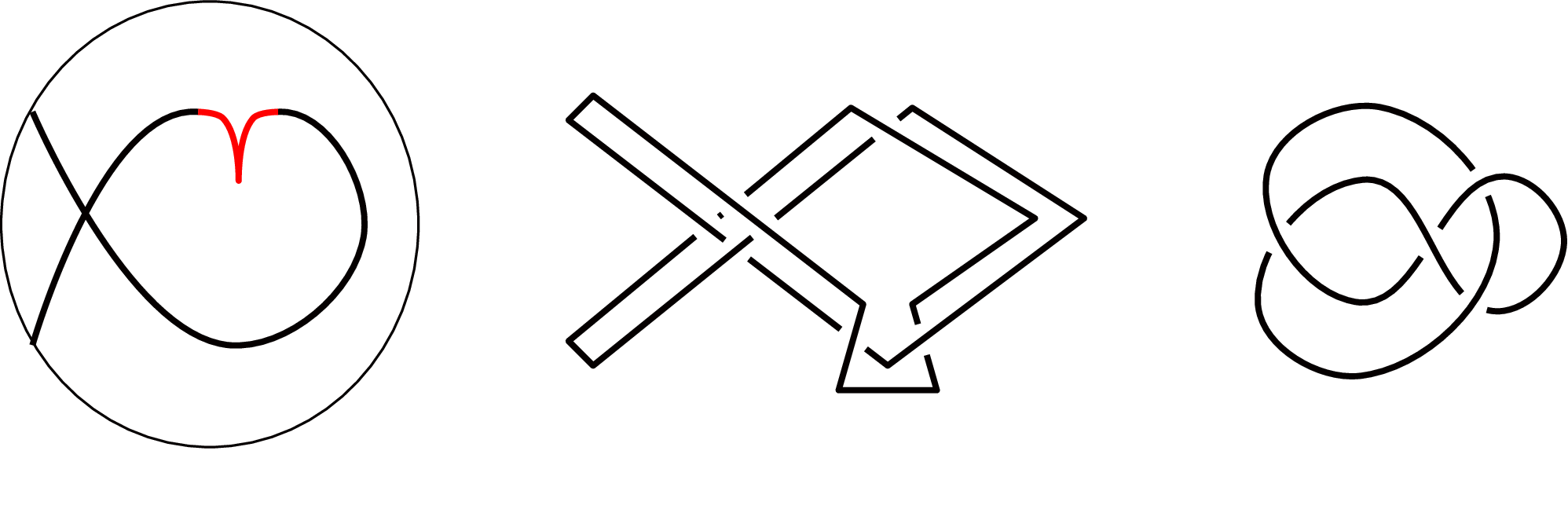}
\caption{The divide with cusps $P$ and the link $L_P$ of $P$.}
\label{figure:example_divide_with_cusps}
\begin{picture}(400,0)(0,0)
\put(285,97){$\approx$}
\put(60,35){$P$}
\put(210,35){$L_{P}$}
\end{picture}
\end{figure}
\end{example}

\subsection{Shadows}
\label{subsection:shadows}

\begin{definition}
A compact space $X$ is called a {\it simple polyhedron} if a regular neighborhood of each point in $X$ is homeomorphic to one of the five local models shown in Figure \ref{figure:local_model_of_simple_polyhedron}. 
\begin{figure}[htbp]
\centering\includegraphics[width=13cm]{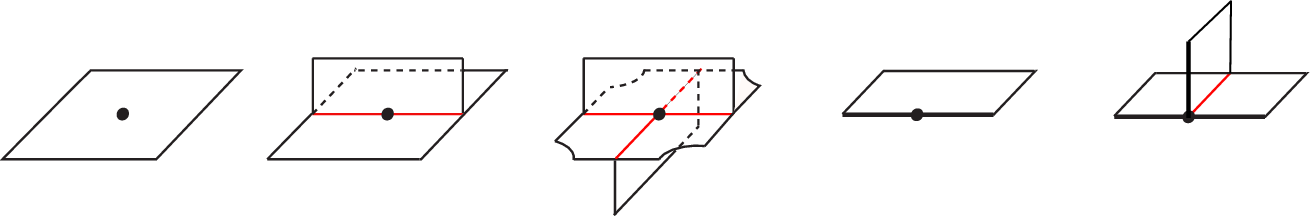}
\begin{picture}(400,0)(0,0)
\put(40,0){(i)}
\put(110,0){(ii)}
\put(186,0){(iii)}
\put(265,0){(iv)}
\put(342,0){(v)}
\end{picture}
\caption{The local models of a simple polyhedron.}
\label{figure:local_model_of_simple_polyhedron}
\end{figure}
A point $x \in X$ whose regular neighborhood is homeomorphic to the model $\mathrm{(iii)}$ is called a {\it vertex} of $X$. 
The set of points whose regular neighborhoods are homeomorphic to the models $\mathrm{(ii)}$ or $\mathrm{(iii)}$ is called the {\it singular set} of $X$ and is denoted by $SX$. 
The set of points whose regular neighborhoods are homeomorphic to the models $\mathrm{(iv)}$ or $\mathrm{(v)}$ is called the {\it boundary} of $X$ and is denoted by $\partial X$. 
A connected component of $X \setminus SX$ is called a {\it region} of $X$.
A region of $X$ that has no intersection with $\partial X$ is called an {\it internal region} of $X$. 
Otherwise, a region is called an {\it external region} of $X$. 
\end{definition}

\begin{definition}
Let $M$ be an orientable closed $3$-manifold and $G$ be a trivalent graph in $M$.  
A simple polyhedron $X$ properly embedded in a compact oriented $4$-manifold $W$ is called a {\it shadow} of $(M, G)$ (or simply a {\it shadow} of $M$) if it satisfies the following conditions:  
\begin{enumerate}
\item $W$ collapses onto $X$.
\item $X$ is locally flat in $W$, i.e., for every point $x$ in $X$, there exists a local coordinate neighborhood $(U, \varphi)$ of $x$ in $W$ such that $\varphi(U \cap X) \subset \mathbb{R}^{3} \subset \mathbb{R}^{4}$.
\item $(M, G) = (\partial W, \partial X)$.
\end{enumerate}
\end{definition}

It is well known that any pair of a closed $3$-manifold and a trivalent graph in it has a shadow (Turaev \cite{Tur94}).

A {\it gleam} $gl$ on a simple polyhedron $X$ is a map from the set of internal regions of $X$ to the set $\frac{1}{2} \ZZ$ of half-integers satisfying a certain condition (Turaev \cite{Tur94}). A simple polyhedron equipped with a gleam is called a {\it shadowed polyhedron}.
\begin{theorem}[Turaev \cite{Tur94}]
\label{theo:turaev_reconstruction}
\begin{itemize}
\item[(i)] For a shadowed polyhedron $(X, gl)$, there exists a canonical way to uniquely construct, up to diffeomorphism, a compact oriented $4$-manifold $W_X$ and an embedding $X \hookrightarrow W_X$ so that $X \subset W_X$ is a shadow of $\partial W_X$.
\item[(ii)] If a simple polyhedron $X$ embedded in a compact oriented $4$-manifold $W$ is a shadow of $(\partial W, \partial X)$, then $X$ is canonically assigned a gleam $gl$ such that $W$ is reconstructed by applying the way of $(\mathrm{i})$ to the shadowed polyhedron $(X, gl)$.
\end{itemize}
\end{theorem}
\begin{definition}
The canonical way of constructing $W_X$ from $X$ in (i) of Theorem \ref{theo:turaev_reconstruction} 
is called {\it Turaev's reconstruction}. 
\end{definition}
We briefly explain (ii) of Theorem \ref{theo:turaev_reconstruction}, that is, how a gleam $gl$ is canonically assigned to a shadow $X$ of $(\partial W,\partial X)$. The following description is based on Martelli \cite{M11}. 
Let $R$ be an internal region of $X$ and $R'$ a compact surface whose interior is homeomorphic to $R$. 
Let $\tau : R' \to X$ be the continuous map that is an extension of the inclusion $R  \hookrightarrow X$.
For simplicity, we assume that $\tau$ is injective. 
For the assignment of gleams in the general case, see e.g. \cite{CM17, CT08, M11}. 
For a point $x \in SX$ of the singular set $SX$ of $X$, let $U_x$ be a neighborhood of $x$ in $W$ that gives the local flatness of the shadow $X \subset W$. 
Since $SX$ is compact, there exists a finite number of points $x_1, x_2, \ldots , x_k$ such that $SX \subset \bigcup_i U_{x_i}$. 
Then, we set $\overline{R} :=(X \setminus \bigcup_i \Int U_{x_i}) \cap R$. 
Let $\pi : W \to X$ be the projection obtained from the collapsing $W \searrow X$. 
We may assume without loss of generality that $W_{R} := \pi^{-1}(\overline{R})$ is a $D^{2}$-bundle over $\overline{R}$.  
Then, the boundary of $W_{R}$ can be decomposed as $\partial W_{R} = \partial_{h}W_{R} \cup \partial_{v}W_{R}$,  
where $\partial_{h}W_{R}$ is an $S^{1}$-bundle over $\overline{R}$ and $\partial_{v}W_{R}$ is a $D^{2}$-bundle over $\partial \overline{R}$.  
We give the orientation $\partial_{h}W_{R}$ induced by the orientation of the boundary of $W_{R}$. 
We fix a section $s$ of the $S^{1}$-bundle $\partial_{h}W_{R}$ over $\overline{R}$ and an orientation of the $S^1$-fiber. 
For each boundary torus $T_{i}$ of $\partial_{h}W_{R}$, we take a meridian and longitude $(\mu_{i}, \lambda_{i})$ of $T_{i}$, where $\mu_{i}$ is a boundary component of the section $s$ and $\lambda_{i}$ is the $S^{1}$-fiber of $\partial_{h}W_{R}$. 
Then, we give the orientation of  $\lambda_{i}$ induced by the orientation of the $S^1$-fiber, and the orientation of $\mu_{i}$ in such a way that $(\mu_{i}, \lambda_{i})$ defines a positive homology basis of $T_{i}$.

In the above setting, the gleam $gl(R) \in \frac{1}{2}\ZZ$ is assigned to the region $R$ as follows. 
Let $r_{i}$ be a component of $\partial \overline{R}$
and $T_{i}$ be the boundary torus of $\partial_{h}W_{R}$ corresponding to $r_{i}$.
We take a point $x$ on $r_{i}$.  
Let $V_x$ be a neighborhood of $x$ in $W$ that satisfies $V_x \subset U_x$. 
Since the neighborhood $U_x$ gives the locally flatness of the shadow $X \subset W$, 
there exists a $3$-ball $D_{x}^3$ centered at $x$ such that $V_x \cap X \subset D_{x}^3$. 
Then, we take the intersection $D_{x}^1 \subset W$ of the $3$-ball $D_{x}^{3}$ and the fiber $D^2 = \pi^{-1}(x)$ over $x$, which is an interval in $W$ centered at $x$. 
By taking such intervals for each point of $r_{i}$, we define a $D^{1}$-bundle over $r_{i}$.  
Then, we defines an $S^{0}$-bundle over $r_{i}$ by taking the boundaries of the fibers of the $D^{1}$-bundle over $r_{i}$. 
If the $S^{0}$-bundle is trivial, it consists of two parallel simple closed curves on $T_i$ expressed as $\mu_{i} + n \lambda_{i}$ for some integer $n \in \mathbb{Z}$.  
In this case, we set $g_{i} := n$.  
If the $S^{0}$-bundle is non-trivial, it consists of a single simple closed curve on $T_{i}$ expressed as $2\mu_{i} + n\lambda_{i}$ for some integer $n \in \mathbb{Z}$. 
In this case, we set $g_{i} := \frac{n}{2}$. 
For all boundary components $r_{i}$ of $\overline{R}$, we define a half-integer $g_{i}$ as above. Then, we assign the gleam $gl(R)$ to $R$ by setting $gl(R) := \sum_{i} g_{i}$.

\begin{example}[Turaev \cite{Tur94}, Costantino-Thurston \cite{CT08}]
\label{ex:shadows_of_DG}

Let $G$ be a trivalent graph in $S^{3}$ and $DG$ be its diagram.  
Then, a shadow of $(S^{3}, G)$ can be constructed from the diagram $DG$ as follows. 
We regard $S^{3}$ as the boundary of $D^{4}$. 
Let $\pi : S^{3} \to D$ be the projection obtained by restricting the collapsing $D^{4} \searrow D$ to $S^{3}=\partial D^{4}$. 
We set $D' := D - \Int \Nbd(\partial D; D)$. 
Then, $\pi^{-1}(D')$ is homeomorphic to the solid torus $S^{1} \times D'$.  
We may assume by isotopy so that $G$ lies in a sufficiently small neighborhood $[-\varepsilon, \varepsilon] \times D'$ of $\{0\} \times D'$, and that $\pi$ is generic with respect to $G$.  
Let $X_{DG}$ be the mapping cylinder of the projection $\pi$:
$$X_{DG} := (G \times [-1, 1]) \sqcup D / (x, 0) \sim \pi(x).$$
We can naturally regard $X_{DG}$ as a simple polyhedron in $D^{4}$. Then, $X_{DG}$ is a shadow of $(S^{3}, G)$. 
The local contribution to the gleam of $X_G$ around a crossing of the diagram $DG$ is as shown in Figure \ref{figure:local_contribution_Turaev}.
The gleam of the shadow $X_{DG}$ is given by the sum of these local contributions around the crossings of $DG$.
\begin{figure}[htbp]
\centering\includegraphics[width=3cm]{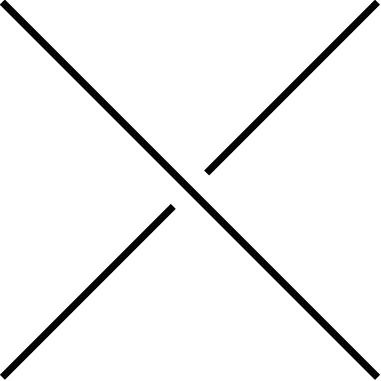}
\caption{The local contribution around a crossing of the diagram $DG$.}
\label{figure:local_contribution_Turaev}
\begin{picture}(400,0)(0,0)
\put(155,75){$-\frac{1}{2}$}
\put(225,75){$-\frac{1}{2}$}
\put(196,105){$\frac{1}{2}$}
\put(196,40){$\frac{1}{2}$}
\end{picture}
\end{figure}
\end{example}

\subsection{Divides with gleams}
\label{subsection:divides_with_gleams}
In this section, we extend the notion of gleams from a shadowed polyhedron to a A'Campo's divide and define a {\it divide with gleams}, which is a divide whose regions are equipped with gleams. 
Let $P$ be a divide. We set $F_{P} := \{(x, u) \in D^4 \mid x \in P, u \in T_{x}P \}$, which is an immersed surface in $D^{4}$ consisting of disk or annulus regions. It has self-intersections only at the double points of $P$, and its boundary is the link $L_P$ of the divide $P$. Then, we define $X_{P} := F_{P} \cup D$. 
Let $D^{4} / \iota$ be the quotient space of induced by the involution $\iota(x,u) = (x,-u)$ of $D^4$. 
Let $q : D^{4} \to D^{4} / \iota$ be the quotient map of $D^{4} / \iota$. 
Note that $D^{4} / \iota$ is homeomorphic to $D^{4}$.
We identify $D^{4} / \iota$ with $D^4$. 
Then, the quotient map $q : D^{4} \to D^{4}/\iota \approx D^4$ is the double branched covering branched along the disk $D$ and can be expressed as \( q(x, re^{\theta i}) = (x, re^{2\theta i}) \). 
Furthermore, the restriction map $q|_{S^{3}} : S^{3} \to S^{3}$ is the double branched covering branched along the boundary $\partial D$ of the disk $D$. 
Then, we define $Y_{P} := q(X_{P})$ and $G_{P} := q(L_{P} \cup \partial D)$. 
Since the link $L_P$ either does not intersect $\partial D$ or intersects it transversely, the image $G_{P}$ is a trivalent graph in $S^{3}$.
Then, the following holds.

\begin{proposition}
\label{prop:YP_is_shadow}
Let $P$ be a divide. 
Then, the polyhedron \( Y_{P} \subset D^{4} \) is a shadow of \( (S^{3}, G_{P}) \).
\end{proposition}

\begin{proof}
It is clear that $D^{4}$ collapses to $X_{P}$.
This collapse $D^{4} \searrow X_{P}$ and the double branched covering $q : D^{4} \to D^{4}$
implies that $D^{4} = q(D^4)$ collapses to $Y_{P}$. 
Furthermore, by the definition of $Y_P$, the boundary of $Y_P$ is the trivalent graph $G_P$.
Thus, it is sufficient to show that $Y_{P}$ is a locally flat simple polyhedron in $D^{4}$. 
We take a point $c \in X_{P}$. 
By the definition of $X_{P}$, a neighborhood $\Nbd(c ; X_{P})$ of $c$ is homeomorphic to one of the local models in Figure \ref{figure:local_models_XP}. 
\begin{figure}[htbp]
\centering\includegraphics[width=11cm]{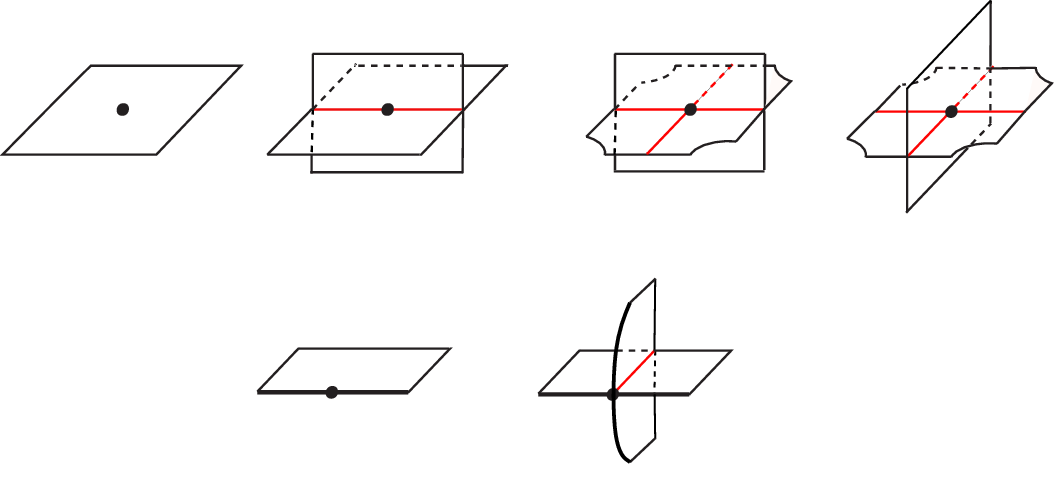}
\caption{The local models of $X_P$.}
\label{figure:local_models_XP}
\begin{picture}(400,0)(0,0)
\put(282,145){$\cup$}

\put(63,107){(i)}
\put(148, 107){(ii)}
\put(275,107){(iii)}
\put(135,30){(iv)}
\put(228,30){(v)}
\end{picture}
\end{figure} 
In what follows, we show that the images under $q$ of the neighborhood models in Figure \ref{figure:local_models_XP} coincide with the neighborhood models of a simple polyhedron shown in Figure \ref{figure:local_model_of_simple_polyhedron}. 
The restriction of the double branched covering $q : D^4 \to D^4, (x,re^{i\theta}) \mapsto (x,re^{2i\theta})$, 
to the $2$-disk $D \subset D^4$ gives the identity map $q|_{D} : D \to D, (x,0) \mapsto (x,0)$. 
Therefore, the images under $q$ of the models in Figure \ref{figure:local_models_XP} (i) and (iv), which are subsets of the $2$--disk $D$, are homeomorphic to the models in Figure \ref{figure:local_model_of_simple_polyhedron} (i) and (iv), respectively. 
We focus on the remaining models. 
The top row of Figure~\ref{figure:local_image_Imq} depicts the local models (ii), (iii), and (v) of $X_P$.
The vectors in this figure represent the subset $F_P=\{(x,u)\in D^4 \mid x\in P,\; u\in T_xP\}$ of $X_P=F_P\cup D$. 
These vectors consist of the two opposite tangent vectors based at each point of the divide $P$. 
Since the branched covering map $q$ is given by $q(x,re^{i\theta})=(x,re^{2i\theta})$,
the images of these two opposite vectors coincide and become a single vector.
Consequently, the image $q(\mathrm{Nbd}(c;X_P))$ of the neighborhood $\mathrm{Nbd}(c;X_P)$ is depicted in the bottom row of Figure \ref{figure:local_image_Imq}, and it coincides with the local models (ii), (iii), and (v) in Figure \ref{figure:local_model_of_simple_polyhedron}.
\begin{figure}[htbp]
\centering\includegraphics[width=12cm]{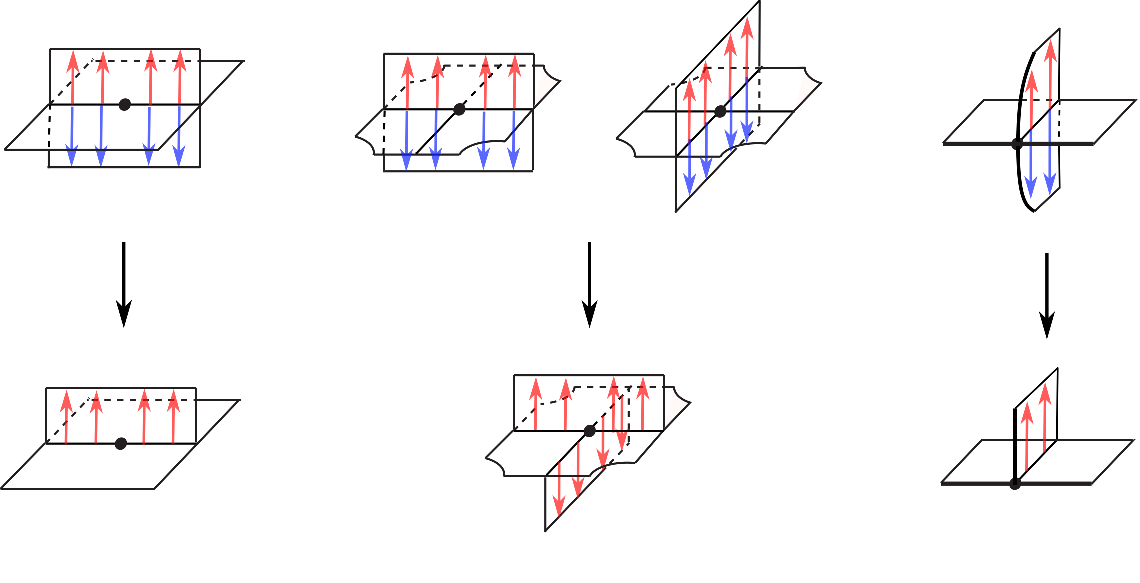}
\caption{The images of the local models (ii), (iii), and (v) of $X_P$ under the double covering $q$ of $D^4$ branched along the disk $D$.}
\label{figure:local_image_Imq}
\begin{picture}(400,0)(0,0)
\put(45,135){$q$}
\put(180,135){$q$}
\put(320,135){$q$}
\put(200,180){$\cup$}
\put(55,43){(ii)}
\put(190,43){(iii)}
\put(330,43){(v)}
\end{picture}
\end{figure}
\end{proof}

\begin{definition}
The trivalent graph $G_P \subset S^3$ is called the {\it trivalent graph of the divide $P$}, and the shadow $Y_P$ of $(S^3, G_P)$ is called the {\it shadow of $(S^3, G_P)$ associated with the divide $P$}. 
Furthermore, the simple polyhedron that is homeomorphic to $Y_P$ is called the {\it simple polyhedron of $P$}.
\end{definition}

For a divide $P$, there exists a natural one-to-one correspondence between the set of internal regions of $P$, $X_{P}$ and $Y_{P}$.
Therefore, unless otherwise stated, we identify an internal region of $P$ with that of $X_{P}$, and $Y_{P}$.

\begin{definition}
Let $P$ be a divide and $gl$ be a map from the set of internal regions of \( P \) to the set \( \frac{1}{2} \mathbb{Z} \) of half-integers.  
Then, the pair $(P, gl)$ is called a {\it divide with gleams} if the pair $(Y_{P}, gl)$ is a shadowed polyhedron. 
Here, $Y_P$ is the simple polyhedron of $P$ and $gl$ is regarded as a map from the set of internal regions of $Y_{P}$ to $\frac{1}{2}\mathbb{Z}$.
\end{definition}



A divide with gleams $(P, gl)$ defines a link $L_{(P,gl)}$ in $S^{3}$ as follows:
Since $Y_{P}$ is contractible, the $4$-manifold obtained from the shadowed polyhedron $(Y_{P}, gl)$ by Turaev's reconstruction is diffeomorphic to the $4$-ball $D^{4}$. Therefore, $Y_P$ can be embedded in $D^4$ so that $Y_P$ is a shadow of $S^3$.
Let $G_{(P, gl)}$ denote the trivalent graph in $S^3$ given by the boundary of the shadow $Y_P$. 
Then, we define a link $L_{(P, gl)}$ in $S^{3}$ as the link obtained by pulling back the trivalent graph $G_{(P,gl)}$ in $S^3$ via the double branched covering $q|_{S^3} : S^3 \to S^3$ branched along $\partial D$.

\begin{definition}
For a divide with gleams $(P, gl)$, 
the trivalent graph $G_{(P,gl)}$ and the link \( L_{(P, gl)} \) in \( S^{3} \) defined above are called 
the {\it trivalent graph of the divide with gleams \( (P, gl) \)} and the {\it link of the divide with gleams \( (P, gl) \)}, respectively.
\end{definition}

As shown in the next proposition, one can assign a suitable gleam $gl$ to a given divide $P$ so that the link of the divide with gleams $(P,gl)$ coincides with the link $L_P$ of the divide $P$.

\begin{proposition}
\label{prop:canonical_gleam}
Let $P$ be a divide. Then, there exists a divide with gleams \((P', gl)\) satisfying the following conditions:
\begin{itemize}
\item \(P' = P\),
\item For any internal region \(R\) of $P$, \(gl(R) = \frac{4-n}{2} \in \frac{1}{2}\mathbb{Z}\). Here, \(n\) is the number of times the boundary of the closure of $R$ passes through double points of $P$.
\item $L_{(P', gl)}$ is equivalent to  $L_{P}$.
\end{itemize}
\end{proposition}

\begin{proof}
Let $G_P$ be the trivalent graph of the divide $P$ and $Y_P$ be the shadow of $(S^3, G_P)$ associated with the divide $P$. 
If the gleam $gl$ of $Y_P$ coincides with the gleam defined in Proposition~\ref{prop:canonical_gleam}, then by Turaev's reconstruction, the trivalent graph $G_{(P,gl)}$ of the divide with gleams $(P,gl)$ is equivalent to $G_P$. 
It immediately follows that the link $L_{(P,gl)}$ of $(P, gl)$ is equivalent to $L_P$.
Therefore, in order to prove Proposition~\ref{prop:canonical_gleam}, it suffices to show that the gleam $gl$ of $Y_P$ coincides with the gleam defined in Proposition~\ref{prop:canonical_gleam}.
Let \(R\) be an internal region of \(Y_{P}\) and $n$ be the number of times the boundary of the closure of $R$ passes through vertices of $Y_{P}$, equivalently the double points of $P$. 
We now assign the gleam of $R$ according to (ii) of Theorem \ref{theo:turaev_reconstruction} explained in Section \ref{subsection:shadows}. 
That is, for the internal region $R$, we take the simple closed curve on the boundary torus $T_{R}$ of $\partial _{h}W_{R}$ defined by the $S^{0}$-bundle over $\partial \overline{R}$ and express it in terms of the homology basis $(\mu, \lambda)$ of $T_{R}$. 
Without loss of generality, 
we may assume that the region $R$ is bounded by the edges of $P$ with slope $1$ or $-1$, 
and that $\partial R$ has $n$ vertices aligned in a row. 
\begin{figure}[htbp]
\centering\includegraphics[width=7cm]{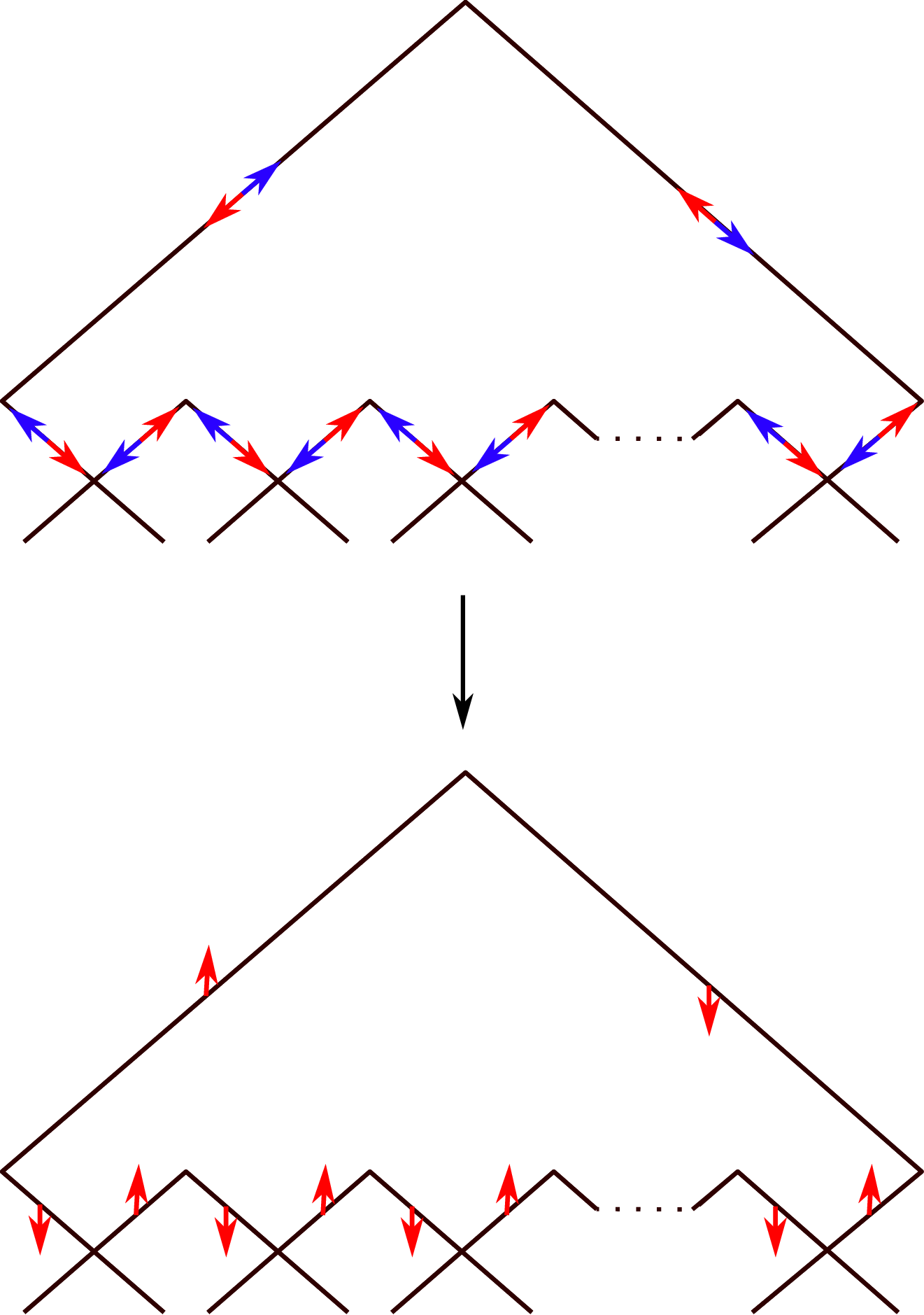}
\begin{picture}(400,0)(0,0)
\put(115,178){$c_1$}
\put(155,178){$c_2$}
\put(195,178){$c_3$}
\put(275,178){$c_n$}
\put(194,238){$R$}

\put(205,155){$q$}

\put(116,12){$c_1$}
\put(156,12){$c_2$}
\put(195,12){$c_3$}
\put(275,12){$c_n$}
\put(194,72){$R$}

\end{picture}
\caption{The neighborhood of $\Nbd(R ; X_P)$ and its image under the double branched covering $q$.}
\label{figure:canonical_gleam}
\end{figure}
Then, a regular neighborhood $\Nbd(R ; X_{P})$ of $R$ in $X_P$ can be drawn as shown in the top of Figure \ref{figure:canonical_gleam}. 
Here, the vector fields in the top of Figure~\ref{figure:convention_cusp} represents the subset $F_P=\{(x,u)\in D^4 \mid x\in P,\; u\in T_xP\}$ of $X_P=F_P\cup D$.
These vector fields consist of pairs of opposite tangent vectors at each point of the divide $P$. 
The vectors rotate by $-\frac{\pi}{2}$ along the edge from the vertex $c_i$ to $c_{i+1}$ $(i \neq n)$, and by $\frac{3\pi}{2}$ along the edge from $c_n$ to $c_1$. 
Since the double branched covering $q : D^4 \to D^4$ is given by $q(x, re^{i\theta}) = (x, re^{2i\theta})$, the images of each pair of opposite vectors coincide and become a single vector, and the image $q(\Nbd(R;X_P)) \subset Y_P$ is depicted in the bottom of Figure~\ref{figure:canonical_gleam}. 
That is, the vectors of the vector field on the bottom of Figure \ref{figure:canonical_gleam} rotate by $-\pi$ along the edge from vertex $c_{i}$ to $c_{i+1}$~$(i \neq n)$ and by $3\pi$ along the edge from $c_{n}$ to $c_{1}$. 
Therefore, by summing these rotations, we find that the vectors rotate by $(4-n)\pi$ along the boundary of the closure of $R$. 
Then, the simple closed curve on $T_{R}$ induced by the $S^{0}$-bundle over $\partial \overline{R}$ is expressed as:
\[
\begin{cases}
\mu + \big(\frac{4-n}{2}\big)\lambda & \text{if } n \in 2\mathbb{Z}, \\
2\mu + (4-n)\lambda & \text{otherwise}.
\end{cases}
\]
From the above expression, the gleam of $R$ is $\frac{4-n}{2}$. Hence, the gleam of $Y_P$ coincides with the gleam defined in Proposition \ref{prop:canonical_gleam}. 
\end{proof}

\begin{definition}
For a divide $P$, the divide with gleams that satisfies the conditions of Proposition \ref{prop:canonical_gleam} is called the {\it canonical divide with gleams of $P$} and denoted by $(P, gl_{P})$. Furthermore, the gleam $gl_{P}$ is called the {\it canonical gleam of $P$}.
\end{definition}

Following the proposition for divides, we next show that a given divide with cusps $P$ can also be equipped with the appropriate gleam so that the resulting divide with gleams has the same associated link as $L_P$. 
Before proving this, we note that the definitions of the trivalent graph $G_P$ of a divide $P$ and the shadow $Y_P$ of $(S^3, G_P)$ associated with $P$ extend naturally to divides with cusps. 
Namely, if $P$ is a divide with cusps, we define
\[
F_P := \{(x,u) \in D^4 \mid x \in P,\; u \in T_xP\}, \quad
X_P := D \cup F_P,
\]
and set
\[
G_P := q(L_P \cup \partial D), \quad
Y_P := q(X_P).
\]
Then $G_P$ is a trivalent graph in $S^3$, and
$Y_P$ is a shadow of $(S^3, G_P)$.
Similarly to the case of divides, the trivalent graph $G_P$ is called the {\it trivalent graph of the divide with cusps $P$}, and the shadow $Y_P$ is called the {\it shadow of $(S^3, G_P)$ associated with the divide with cusps $P$}.

\begin{proposition}
\label{prop:cusp_canonical_gleam}
Let $P$ be a divide with cusps. Then, there exists a divide with gleams \((P', gl)\) satisfying the following conditions:
\begin{itemize}
\item \(P'\) is the divide obtained by removing all cusps from \(P\),
\item For any internal region \(R\) of \(P\), \(gl(R) = \frac{4-n}{2} + m_{1} - m_{2} \in \frac{1}{2}\mathbb{Z}\). Here, \(n\) is the number of times the boundary of the closure of $R$ passes through double points of P, \(m_{1}\) is the number of inner cusps for $R$, and \(m_{2}\) is the number of outer cusps for $R$,
\item $L_{(P', gl)}$ is equivalent to $L_{P}$.
\end{itemize}
\end{proposition}

\begin{proof}
Let $G_P$ be the trivalent graph of the divide with cusps $P$ and $Y_P$ be the shadow of $(S^3, G_P)$ associated with $P$. 
By the definition of $Y_P$, $Y_P$ is the simple polyhedron of the divide $P'$. 
Therefore, it is sufficient to show that the gleam of the shadow \(Y_{P}\) coincides with the gleam defined in Proposition \ref{prop:cusp_canonical_gleam}. 
By Proposition \ref{prop:canonical_gleam}, if $P$ has no cusps, that is, $P$ is a divide, then the gleam of $Y_{P}$ coincides with the gleam defined in Proposition \ref{prop:cusp_canonical_gleam} for $m_{1}=m_{2}=0$. 
We assume that $P$ has at least one cusp, and calculate the local contribution around a cusp of $P$ to the gleam of $Y_{P}$.  
Let $R$ be an internal region of $P$, and $c$ be an inner cusp for $R$. 
In a regular neighborhood $\Nbd(c;X_P)$ of the inner cusp $c$ for $R$, the $2$-polyhedron $X_{P}$ is depicted in the top of Figure \ref{figure:convention_cusp}. 
\begin{figure}[htbp]
\centering\includegraphics[width=6cm]{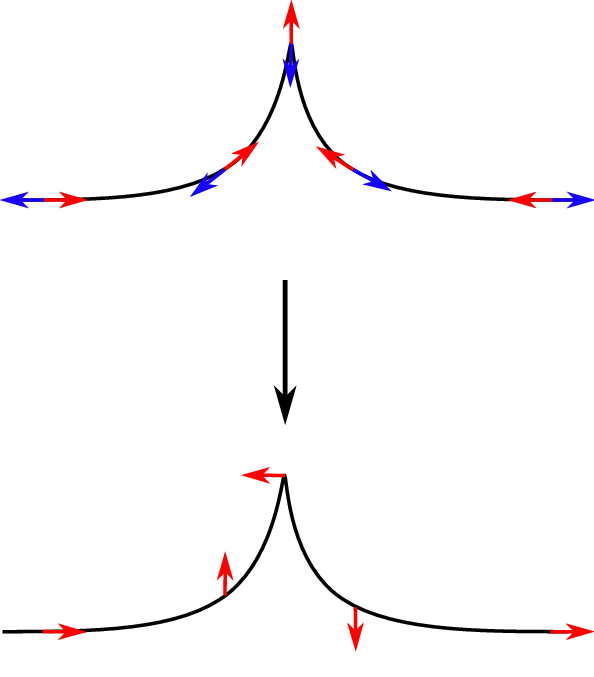}
\begin{picture}(400,0)(0,0)
\put(148, 185){$R$}
\put(148, 60){$R$}
\put(205, 108){$q$}
\end{picture}
\caption{The neighborhood $\Nbd(c ; X_P)$ of the inncer cusp $c$ for $R$ and its image under the double branched covering $q$.}
\label{figure:convention_cusp}
\end{figure}
Here, the vector field on the top of Figure \ref{figure:convention_cusp} represents the subset $F_P=\{(x,u)\in D^4 \mid x\in P,\; u\in T_xP\}$ of $X_P=F_P\cup D$. 
The vectors of the vector field rotate by $\pi$ when passing through the inner cusp $c$ along the boundary of the closure of $R$. 
Then, the image $q(\Nbd(c ; X_P)) \subset Y_P$ is depicted on the bottom of Figure \ref{figure:convention_cusp}. 
Since the double branched covering $q: D^{4} \to D^{4}$ is expressed as $q(x, re^{\theta i}) = (x, re^{2\theta i})$, the vectors on the bottom of Figure \ref{figure:convention_cusp} rotate by $2\pi$ when passing through the inner cusp $c$ along the boundary of the closure of $R$. 
Therefore, the local contribution around the inner cusp $c$ to the region $R$ is $1$. 
For an outer cusp of $R$ as well, similarly to the above, we can show that the local contribution around the outer cusp to the region $R$ is $-1$. 
All of the above implies that the gleam of $Y_{P}$ coincides with the gleam defined in Proposition \ref{prop:cusp_canonical_gleam}.
\end{proof}

\begin{example}

Let $P_{1}$ be the divide shown on the left-top in Figure \ref{figure:example_divide_with_gleams_1}. 
Then, the link $L_{P_{1}}$ of $P_{1}$ is the trefoil knot $3_{1}$ (This can be easily checked using the algorithm for drawing the link of a divide presented by Hirasawa \cite{Hira02}). 
By Proposition \ref{prop:canonical_gleam}, the canonical divide with gleams $(P_{1}, gl_{P_{1}})$ of $P_{1}$ is shown on the right-top in Figure \ref{figure:example_divide_with_gleams_1}. 
Let $P_{1}^{\ast}$ be the divide with cusps shown on the left-bottom in Figure \ref{figure:example_divide_with_gleams_1}. 
Applying Proposition \ref{prop:cusp_canonical_gleam} to $P_{1}^{\ast}$, we obtain a divide with gleams $(P_{1}, gl_{P_{1}^{\ast}})$, which is shown on the right-bottom in Figure \ref{figure:example_divide_with_gleams_1}. 
Since $gl_{P_{1}^{\ast}} = - gl_{P_{1}}$, the link $L_{P_{1}^{\ast}}$ of $P_{1}^{\ast}$ is the mirror image of $L_{P_{1}}=3_{1}$.
\begin{figure}[htbp]
\centering\includegraphics[width=9cm]{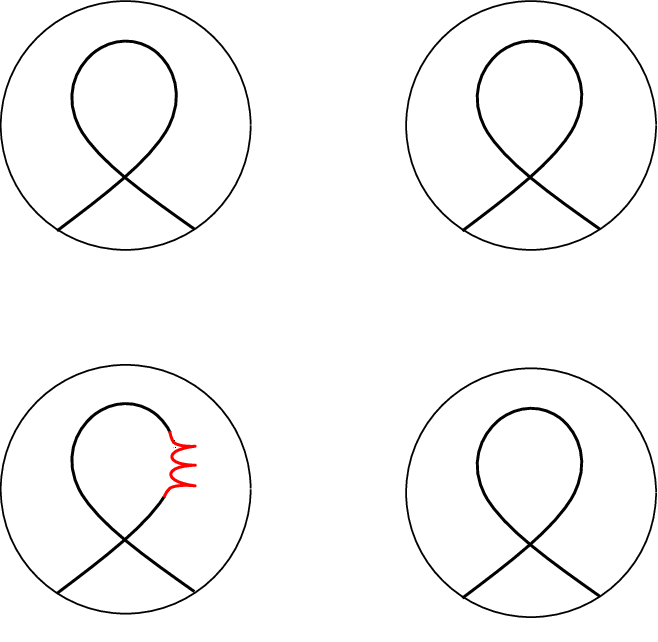}
\begin{picture}(400,0)(0,0)
\put(115, 0){$P_1^{\ast}$}
\put(115, 140){$P_1$}
\put(257, 0){$(P_1, gl_{P_1^{\ast}})$}
\put(258, 140){$(P_1, gl_{P_1})$}

\put(195, 67){$=$}
\put(195, 210){$=$}
\put(270, 75){$-\frac{3}{2}$}
\put(275, 220){$\frac{3}{2}$}

\end{picture}

\caption{The divide $P_{1}$, the divide with cusps $P_{1}^{\ast}$, and their corresponding divides with gleams.}
\label{figure:example_divide_with_gleams_1}
\end{figure}
\end{example}

\begin{example}

Let $P_2$ be the divide with cusps shown on the left in Figure \ref{figure:example_divide_with_gleams_2}.
By Theorem \ref{theo:sugawara}, the link $L_{P_2}$ of $P_2$ is a $2$-periodic knot.
The divide with gleams $(P_2 ', gl)$ obtained by Proposition \ref{prop:cusp_canonical_gleam} for $P_{2}$ is shown on the right in Figure \ref{figure:example_divide_with_gleams_2}.
\begin{figure}[htbp]
\centering\includegraphics[width=10cm]{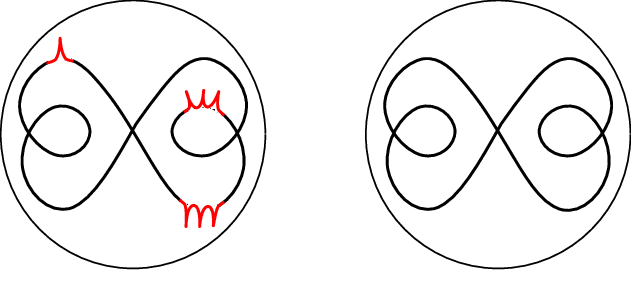}
\caption{The divide with cusps $P_2$ and its corresponding divide with gleams $(P_2 ', gl)$.}
\label{figure:example_divide_with_gleams_2}
\begin{picture}(400,0)(0,0)
\put(110, 45){$P_2$}
\put(264, 45){$(P_2 ', gl)$}

\put(246, 122){$\frac{3}{2}$}
\put(243, 100){$-\frac{1}{2}$}
\put(303, 122){$-\frac{3}{2}$}
\put(310, 100){$\frac{1}{2}$}
\put(195, 117){$=$}
\end{picture}
\end{figure}
Since $(P_2 ', gl)$ is isotopic to $(P_2 ', -gl)$ by rotating $(P_2 ', gl)$ around the origin of the $2$-disk $D$ by $\pi$, the link $L_{P_2}$ is isotopic to its mirror image. 
That is, $L_{P_2}$ is amphicheiral. 
\end{example}

\section{Proof of Main Theorem}
\label{section:proof_of_main_theorem}
In this section, we provide the proof of Theorem \ref{theo:main}, 
which is the main theorem of this paper. 
First, we introduce some notation and assumptions. 
We regard the $3$-sphere $S^3$ as the boundary of the $4$-ball $D^4$. 
We define a $2$-sphere $S^{2}$ in $S^{3}$ by
\[
S^{2} := \left\{ (x,u) \in S^{3} \mid \arg(u) = \frac{3\pi}{4} \right\} 
\cup 
\left\{ (x,u) \in S^{3} \mid \arg(u) = -\frac{\pi}{4} \right\}.
\]
Let $\iota$ be the involution of $D^4$ defined by $\iota(x,u) = (x, -u)$. 
Note that the fixed point set of $\iota$ coincides with the great circle $\{(x, u) \in S^{2} \mid |x| = 1\}$ on $S^{2}$.  
Let $L$ be a symmetric link in $S^{3}$. 
We may assume without loss of generality that the symmetric link $L$ satisfies:
\begin{itemize}
\item $\iota(L) = L$,
\item $L \subset S^{2} \times [-\varepsilon, \varepsilon]$ for sufficiently small $\varepsilon > 0$.
\end{itemize}
We consider the projection
\[
S^{2}\times[-\varepsilon,\varepsilon] \longrightarrow S^{2},
\qquad (x,t)\mapsto x.
\]

By perturbing this projection appropriately, we may assume that its restriction to $L$ becomes an immersed curve on $S^2$ with only transverse double points, and moreover that this curve is invariant under the involution $\iota|_{S^2}$. 
By assigning over/under information at each double point of this curve, 
we obtain a transvergent diagram $DL$ of $L$.
We then project $S^{2}$ onto the disk $D$. 
By composing this map with the projection 
$S^{2}\times[-\varepsilon,\varepsilon]\to S^{2}$ defined above,
we obtain a projection
$$\rho : S^{2}\times[-\varepsilon,\varepsilon] \longrightarrow D.$$
Then, we perturb this projection $\rho$ so that no double point of $\rho(L)$ lies on $\partial D$. 
We denote the image of $L$ under this perturbed projection by $P_{DL}$. 
This image $P_{DL}$ is a divide, and it is precisely the divide of the transvergent diagram $DL$ as defined in the Introduction. 
The following proposition shows that one can assign gleams to the regions of the divide $P_{DL}$ so that the resulting divide with gleams has the same associated link as $L$, and moreover that these gleams are naturally determined by the transvergent diagram $DL$.

\begin{proposition}
\label{prop:transvergent_diagrams_to_divides_with_gleams}
Let \( L \) be a symmetric link, and let \( DL \) be its transvergent diagram.  
Then, there exists a divide with gleams $(P, gl)$ satisfying the following conditions:
\begin{itemize}
\item \( P = P_{DL} \),
\item \( gl \) is given by the sum of the local contributions around the crossings of $DL$ shown in Figure \ref{figure:local_contribution},
\item \( L_{(P, gl)} \) is equivalent to \( L \).
\end{itemize}
\begin{figure}[htbp]
\centering\includegraphics[width=13cm]{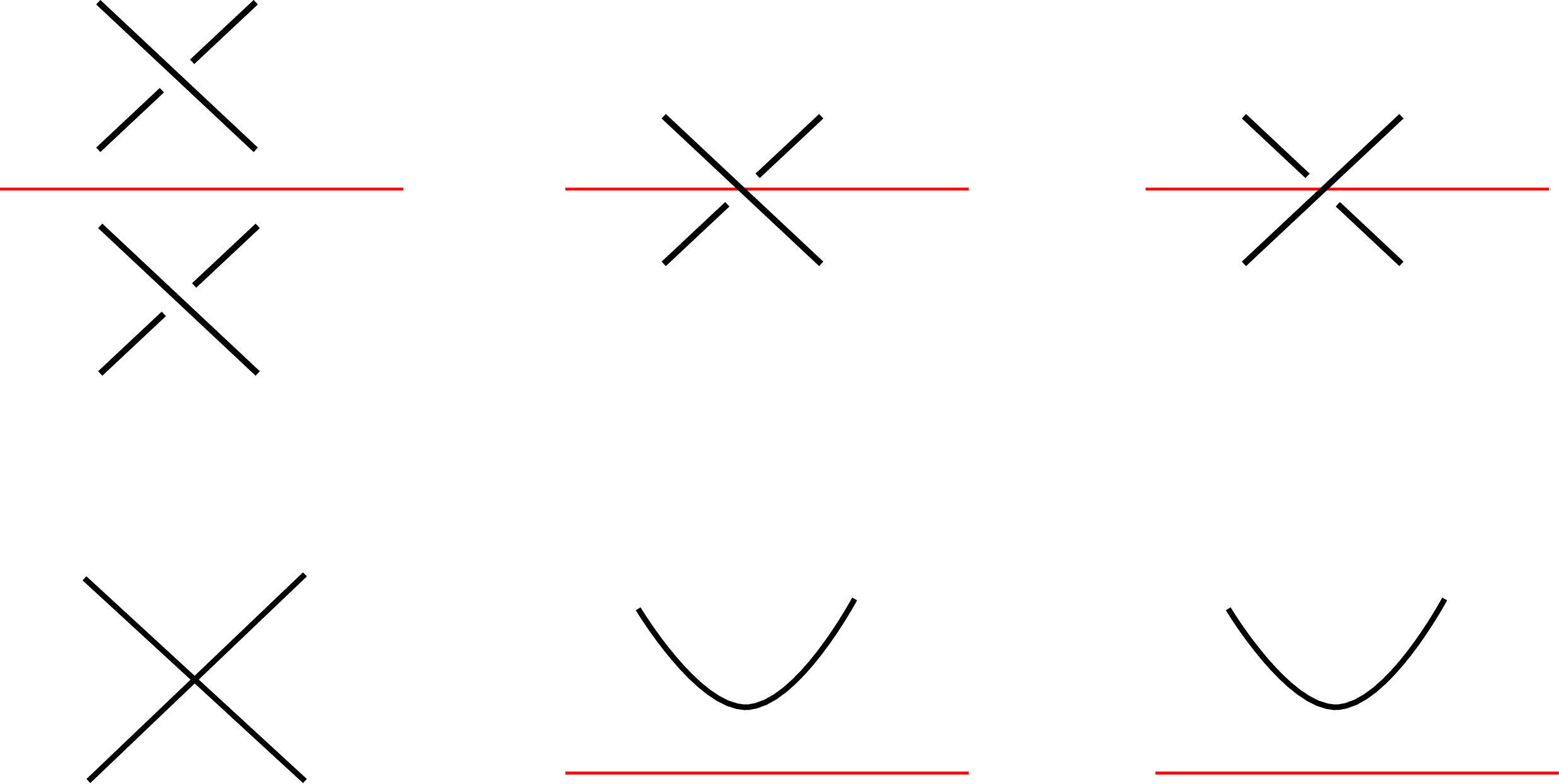}
\caption{The local contributions around crossings of $DL$.}
\label{figure:local_contribution}
\begin{picture}(400,0)(0,0)
\put(0, 210){$DL$}
\put(0, 73){$P_{DL}$}

\put(190, 65){$1$}
\put(323, 65){$-1$}
\put(58, 75){$\frac{1}{2}$}
\put(58, 37){$\frac{1}{2}$}
\put(30, 55){$-\frac{1}{2}$}
\put(77, 55){$-\frac{1}{2}$}

\put(55, 104){\Large{$\downarrow$}}
\put(65, 107){$\rho$}

\put(189, 104){\Large{$\downarrow$}}
\put(199, 107){$\rho$}

\put(326, 104){\Large{$\downarrow$}}
\put(336, 107){$\rho$}
\end{picture}
\end{figure}
\end{proposition}

\begin{example}
\label{ex:transvergent_diagram_to_divide_with_gleams}
Let $L$ be the symmetric link that has the transvergent diagram $DL$ shown on the left-top on Figure \ref{figure:example_cusp_to_gleam}. 
Then, we obtain the divide with gleams $(P_{DL}, gl)$ that satisfies the conditions of Proposition \ref{prop:transvergent_diagrams_to_divides_with_gleams} as shown on the right in Figure \ref{figure:example_cusp_to_gleam}.
\begin{figure}[htbp]
\centering\includegraphics[width=14cm]{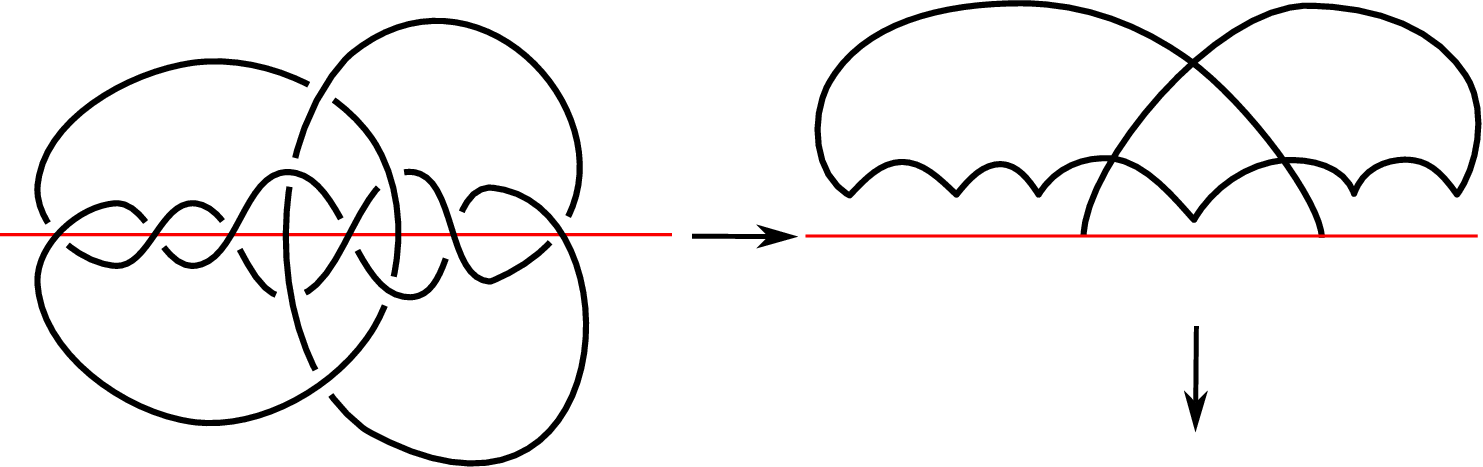}
\centering\includegraphics[width=12cm]{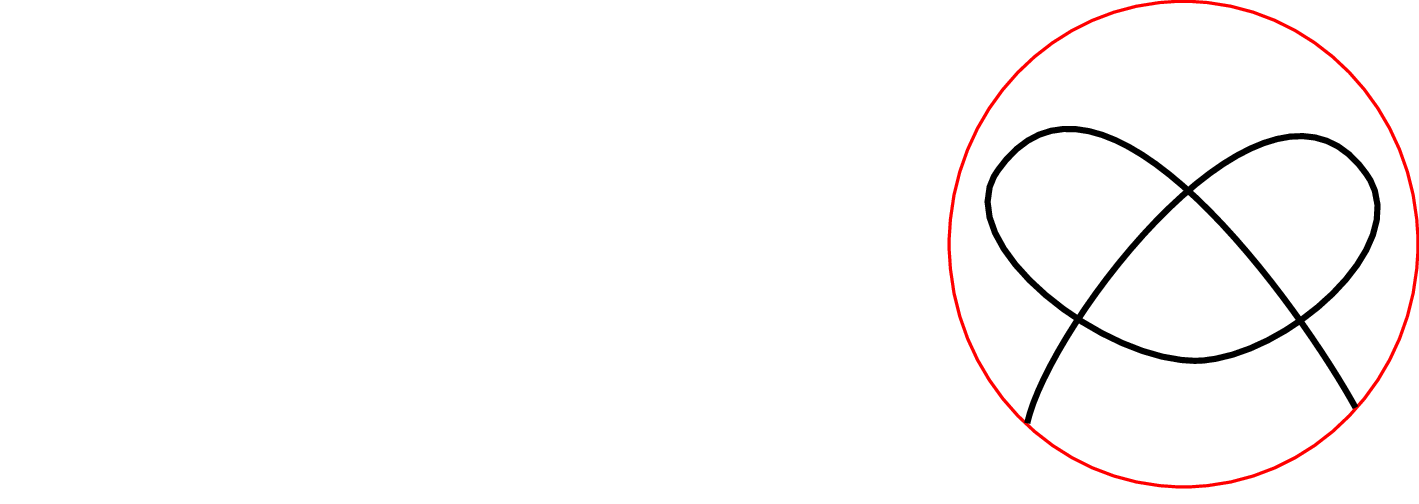}
\begin{picture}(400,0)(0,0)
\put(226, 217){{\scriptsize $-1$}}
\put(252, 217){{\scriptsize $-1$}}
\put(276, 217){{\scriptsize $-1$}}
\put(305, 240){{\scriptsize $\frac{1}{2}$}}
\put(295, 224){{\scriptsize $\frac{1}{2}$}}

\put(314, 232){{\scriptsize $-\frac{1}{2}$}}
\put(326, 220){{\scriptsize $-\frac{1}{2}$}}
\put(304, 219){{\scriptsize $-\frac{1}{2}$}}
\put(316, 209){{\scriptsize $-1$}}

\put(363, 216){{\scriptsize $1$}}
\put(389, 216){{\scriptsize $1$}}
\put(330, 240){{\scriptsize $\frac{1}{2}$}}
\put(345, 223){{\scriptsize $\frac{1}{2}$}}

\put(280, 78){$-2$}
\put(340, 78){$3$}
\put(307, 63){$-\frac{5}{2}$}

\put(75, 120){$DL$}
\put(292, 0){$(P_{DL}, gl)$}
\end{picture}

\caption{The transvergent diagram $DL$ and the divide with gleams $(P_{DL}, gl)$.}
\label{figure:example_cusp_to_gleam}
\end{figure}
\end{example}

\begin{proof}[Proof of Proposition \ref{prop:transvergent_diagrams_to_divides_with_gleams}]
Let $P$ be the divide defined as $P := P_{DL}$. 
Let $gl$ be the gleam of $P$ given by the sum of the local contributions in Figure \ref{figure:local_contribution}. 
Let $q : D^4 \to D^4$ is the double covering of $D^4$ branched along $D$. 
Let $G_L$ be the trivalent graph in $S^3$ defined as $G_L := q(L \cup \partial D)$. 
To prove Proposition \ref{prop:transvergent_diagrams_to_divides_with_gleams}, 
it suffices to show that the link $L$ is equivalent to the link $L_{(P,gl)}$ of $(P,gl)$. 
To show this, it is enough to show that the trivalent graph $G_L$ is equivalent to the trivalent graph $G_{(P,gl)}$ of $(P,gl)$. 
Let $X_{DL}$ be the mapping cylinder of the projection $\rho : S^{2} \times [-\varepsilon, \varepsilon] \to D$ :
$$X_{DL} := \big((L \times [0,1] \sqcup D)\big) / (x, 0) \sim \rho(x).$$ 
$X_{DL}$ can naturally be regarded as a $2$-polyhedron in $D^{4}$.
Then, we set \( Y_{DL} := q(X_{DL}) \). 
Similarly to the proof of Proposition \ref{prop:YP_is_shadow}, $Y_{DL}$ is a shadow of $(S^{3}, G_{L})$. 
Thus, if the gleam \( gl_{DL} \) of the shadow $Y_{DL}$ coincides with the gleam \( gl \), by Turaev's reconstruction, we can show that the two trivalent graphs \( G_{L} \) and \( G_{(P, gl)} \) are equivalent. 
In the following, we show that $gl_{DL}$ and $gl$ coincide. 
By the assumption for a symmetric link we made at the beginning of this section, we can draw a neighborhood of each crossing of $L$ as shown in Figure \ref{figure:cylinder_representations_of_symmetric_links}.  
Here, the four $2$-disks in $S^3$ depicted in Figure \ref{figure:cylinder_representations_of_symmetric_links} are, from top to bottom,
$\{(x, u) \in S^{3} \mid \arg(u) = \frac{3\pi}{4}\}$, 
$\{(x, u) \in S^{3} \mid \arg(u) = \frac{\pi}{4}\}$,
$\{(x, u) \in S^{3} \mid \arg(u) = -\frac{\pi}{4}\}$, and $\{(x, u) \in S^{3} \mid \arg(u) = -\frac{3\pi}{4}\}$.
The images of these crossings under the double branched covering $q|_{S^3} : S^3 \to S^3$ are drawn as shown on the top of Figure \ref{figure:quotient_cylinder_representations_of_symmetric_links}.
\begin{figure}[htbp]
\centering\includegraphics[width=13cm]{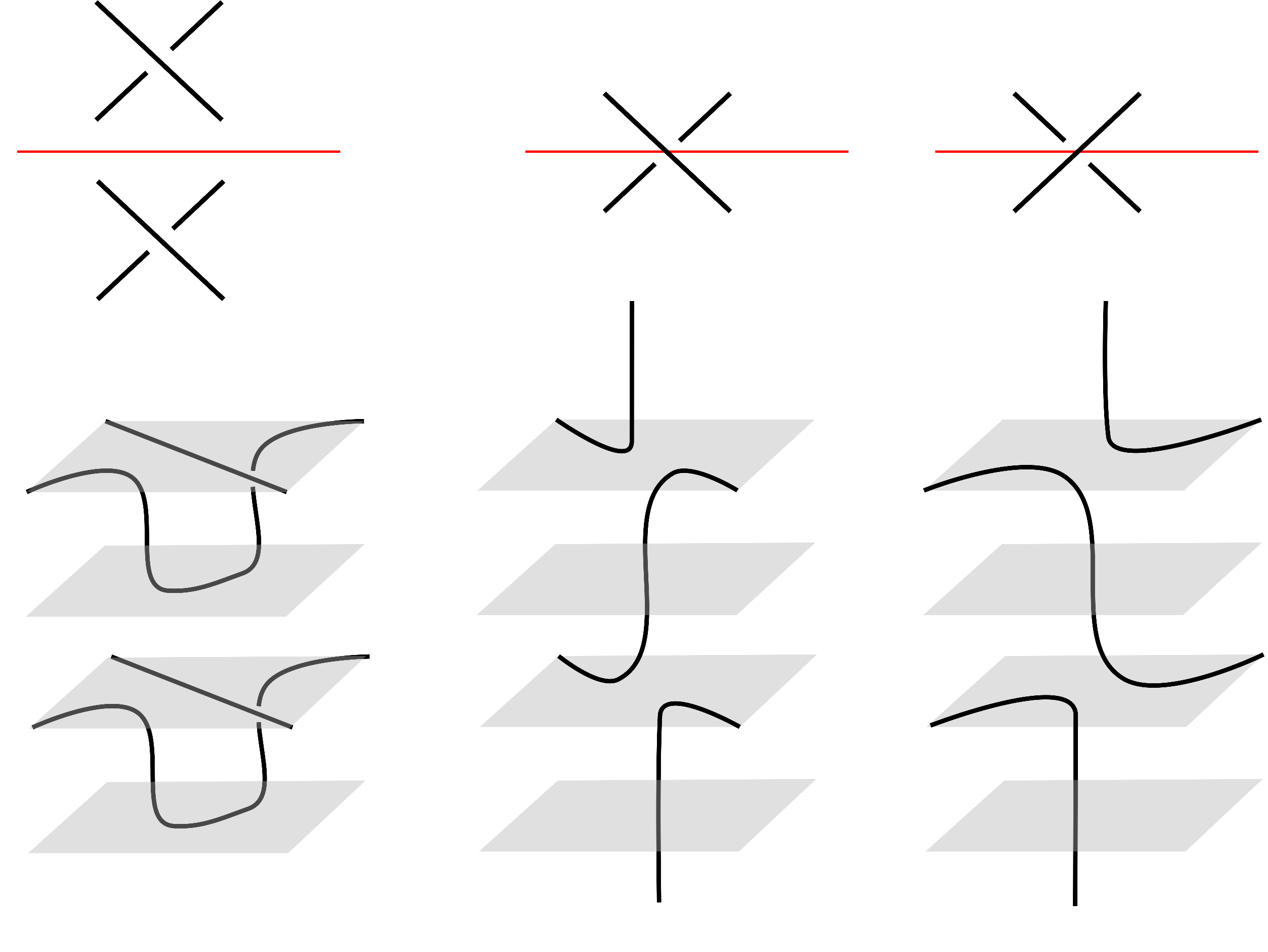}
\caption{The local pictures of crossings of $L$.}
\label{figure:cylinder_representations_of_symmetric_links}
\begin{picture}(400,0)(0,0)
\put(6, 180){$\frac{3\pi}{4}$}
\put(7, 143){$\frac{\pi}{4}$}
\put(0, 110){$-\frac{\pi}{4}$}
\put(0, 70){$-\frac{3\pi}{4}$}
\end{picture}
\end{figure}
Here, the two $2$-disks in $S^3$ depicted in Figure \ref{figure:quotient_cylinder_representations_of_symmetric_links} are $\{(x, u) \in S^{3} \mid \arg(u) = \frac{3\pi}{2}\}$ and $\{(x, u) \in S^{3} \mid \arg(u) = \frac{\pi}{2}\}$. 
Corresponding to these images, $Y_{DL}$ is locally depicted on the bottom of Figure \ref{figure:quotient_cylinder_representations_of_symmetric_links}. 
\begin{figure}[htbp]
\centering\includegraphics[width=12cm]{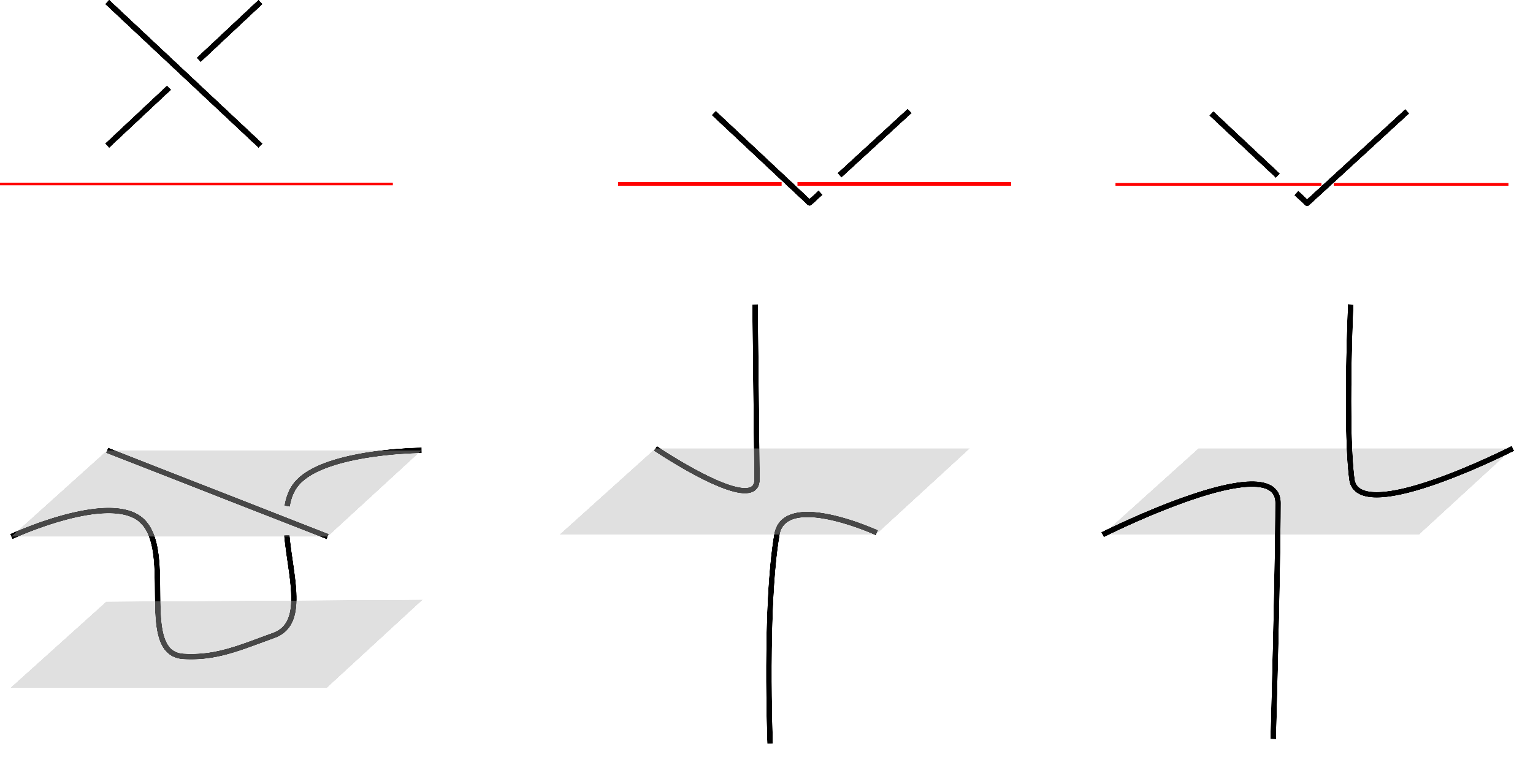}
\centering\includegraphics[width=12cm]{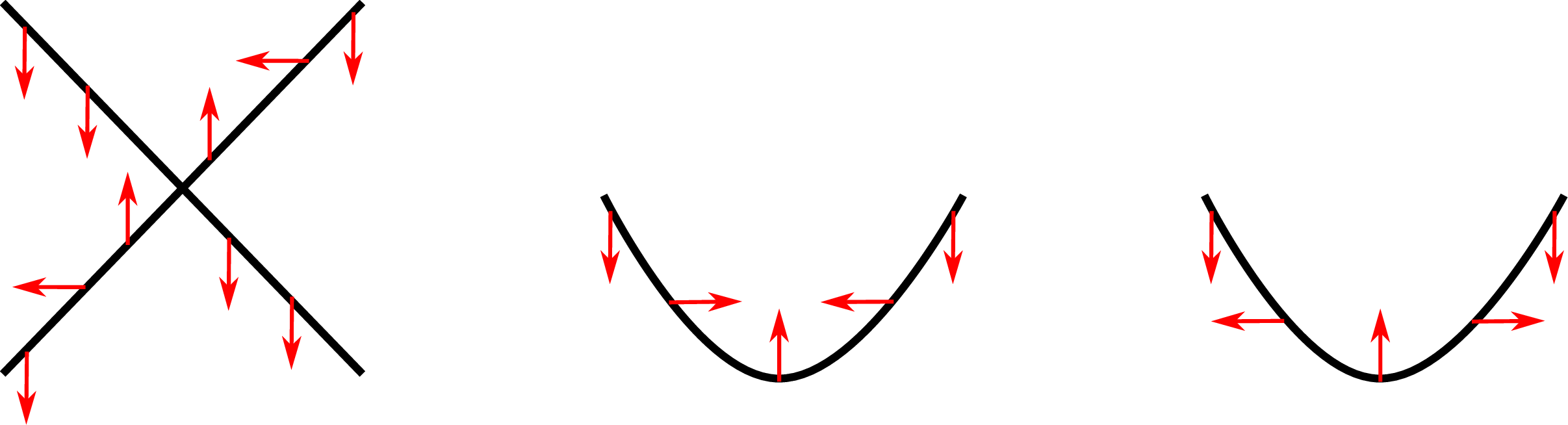}
\caption{The image of each crossing of $L$ under the double branched covering $q|_{S^3}$ and the local representation of $Y_{DL}$ corresponding to these images.}
\label{figure:quotient_cylinder_representations_of_symmetric_links}
\begin{picture}(400,0)(0,0)
\put(12, 223){$\frac{3\pi}{2}$}
\put(12, 183){$\frac{\pi}{2}$}


\put(20, 110){$-\frac{1}{2}$}
\put(90, 110){$-\frac{1}{2}$}
\put(65, 143){$\frac{1}{2}$}
\put(65, 74){$\frac{1}{2}$}

\put(198, 105){$1$}
\put(323, 105){$-1$}
\end{picture}
\end{figure}
Here, the vector field represents the disk or annulus regions of \( Y_{DL} \) attached to \( P \).  
From this figure, we can check that the local contributions around the crossings of $DL$ to the gleam of $Y_{DL}$ are as shown in Figure \ref{figure:local_contribution}.
Hence, \( gl_{DL} \) and \( gl \) are coincide. 
\end{proof}
Next, we show that for a given divide with gleams, one can algorithmically construct a divide with cusps that has the same associated link. 
\begin{proposition}
\label{prop:divides_with_gleams_to_divides_with_cusps}
Let $(P,gl)$ be a divide with gleams.
Then there exists a divide with cusps $P'$ such that:
\begin{itemize}
\item $P'$ is obtained from $P$ by adding finitely many cusps;
\item $L_{P'}$ is equivalent to $L_{(P,gl)}$.
\end{itemize}
\end{proposition}

\begin{proof}
Let $gl_P$ be the canonical gleam of the divide $P$. 
By Proposition~\ref{prop:cusp_canonical_gleam}, adding a cusp along an edge of the divide $P$ changes the gleam of an adjacent region by $+\frac{1}{2}$ if the cusp is an inner cusp for the region, and by $-\frac{1}{2}$ if it is an outer cusp.
Thus it suffices to modify $P$ by adding cusps so that $gl_P$ coincides with $gl$. We construct such a divide inductively as follows. 
Set $P' := P$ and $gl' := gl_P$. Let $G$ be the dual graph of $P'$, whose vertices correspond to regions of $P'$ and whose edges correspond to pairs of adjacent regions of $P'$. Choose one vertex $v_0$ corresponding to an external region of $P'$. Consider a maximal tree $T$ of $G$ containing this vertex $v_0$. 
For each vertex $v$ of $T$ whose corresponding region $R_v$ is an internal region, 
we define
\[
n_v := gl(R_v) - gl'(R_v) \in \mathbb{Z}.
\]
Let $w \neq v_0$ be a vertex of $T$ that is incident to exactly one edge of $T$. 
Let $w_1$ denote the unique vertex of $T$ adjacent to $w$ along that edge.
Adjust the gleams by adding cusps along the edge of the divide
that separates the regions $R_w$ and $R_{w_1}$, 
which corresponding to the vertices $w$ and $w_1$ in the dual graph. 
More precisely, attach $|n_w|$ inner (resp.\ outer) cusps
along this edge on the side of $R_w$
if $n_w \ge 0$ (resp.\ $n_w < 0$).
At this point, the gleam of $R_w$ changes by $n_w$ and that of $R_{w_1}$ changes by $-n_w$. 
Consequently, the resulting divide with cusps satisfies $gl'(R_w)=gl(R_w)$. 
Remove $w$ together with the edge joining $w$ to $w_1$ in the tree $T$. 
Repeat this procedure for the resulting graph. 
Once all vertices of $T$ have been processed, 
remove from $G$ all vertices and edges belonging to $T$. 
If $G$ still has remaining components,
apply the same procedure to each of them. 
Since each step decreases the number of vertices, the process terminates,
and the resulting divide with cusps $P'$ satisfies $gl'=gl$.
\end{proof}

\begin{example}
\label{ex:divide_with_gleams_to_divide_with_cusps}
Let $(P, gl)$ be the divide with gleams shown on the left-top of Figure \ref{figure:algorithm_add_cusps_2}. 
By applying the algorithm in the proof of Proposition \ref{prop:divides_with_gleams_to_divides_with_cusps} to $(P, gl)$, 
we obtain the divide with cusps $P'$ that satisfies the conditions of Proposition \ref{prop:divides_with_gleams_to_divides_with_cusps} as shown in Figure \ref{figure:algorithm_add_cusps_2}.
\begin{figure}[htbp]
\centering\includegraphics[width=8cm]{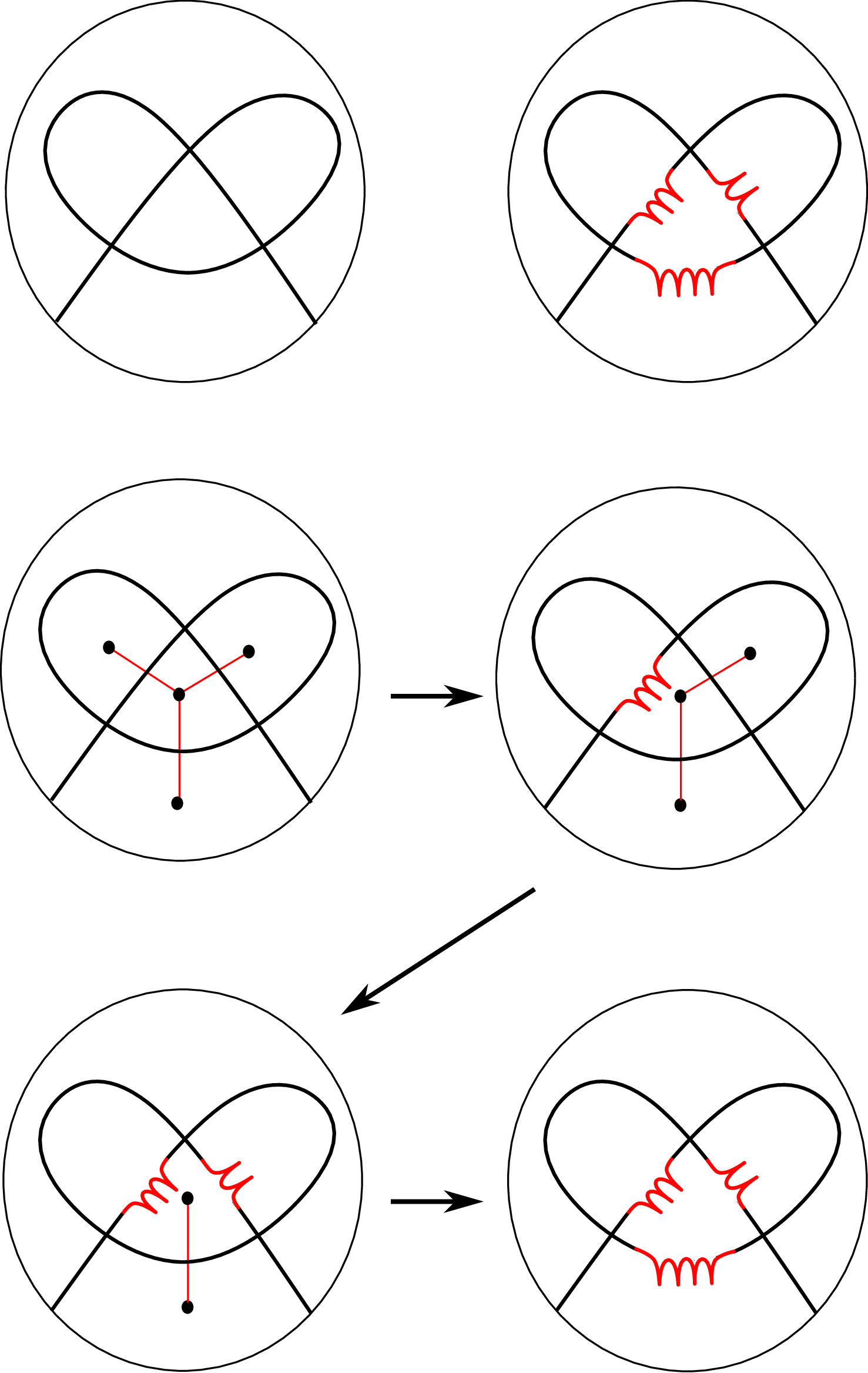}
\begin{picture}(400,0)(0,0)

\put(105, 330){$-2$}
\put(153, 330){$3$}
\put(128, 313){$-\frac{5}{2}$}
\put(119, 255){$(P, gl)$}
\put(261, 255){$P'$}

\put(196, 316){\Large{$=$}}

\put(105, 207){$-3$}
\put(153, 207){$2$}
\put(117.5, 182){$-3$}

\put(240, 204){$0$}
\put(285, 207){$2$}
\put(248, 179){$-6$}

\put(105, 75){$0$}
\put(153, 75){$0$}
\put(117.5, 47){$-4$}

\put(240, 75){$0$}
\put(285, 75){$0$}
\put(265, 50){$0$}

\end{picture}
\caption{The divide with gleams $(P, gl)$ and the divide with cusps $P'$.}
\label{figure:algorithm_add_cusps_2}
\end{figure}
\end{example}

Theorem \ref{theo:main} is proved constructively by using Proposition \ref{prop:transvergent_diagrams_to_divides_with_gleams} and \ref{prop:divides_with_gleams_to_divides_with_cusps} as follows.
\begin{proof}[Proof of Theorem \ref{theo:main}]
By applying Proposition \ref{prop:transvergent_diagrams_to_divides_with_gleams} to the transvergent diagram $DL$ of $L$, we obtain a divide with gleams $(P_{DL}, gl)$.
Then, by applying Proposition \ref{prop:divides_with_gleams_to_divides_with_cusps} to the divide with gleams $(P_{DL}, gl)$, 
we obtain a divide with cusps $P$ that satisfies the conditions of Theorem \ref{theo:main} (see Figure \ref{figure:example_main_theorem}). 
\begin{figure}[htbp]
\centering\includegraphics[width=11cm]{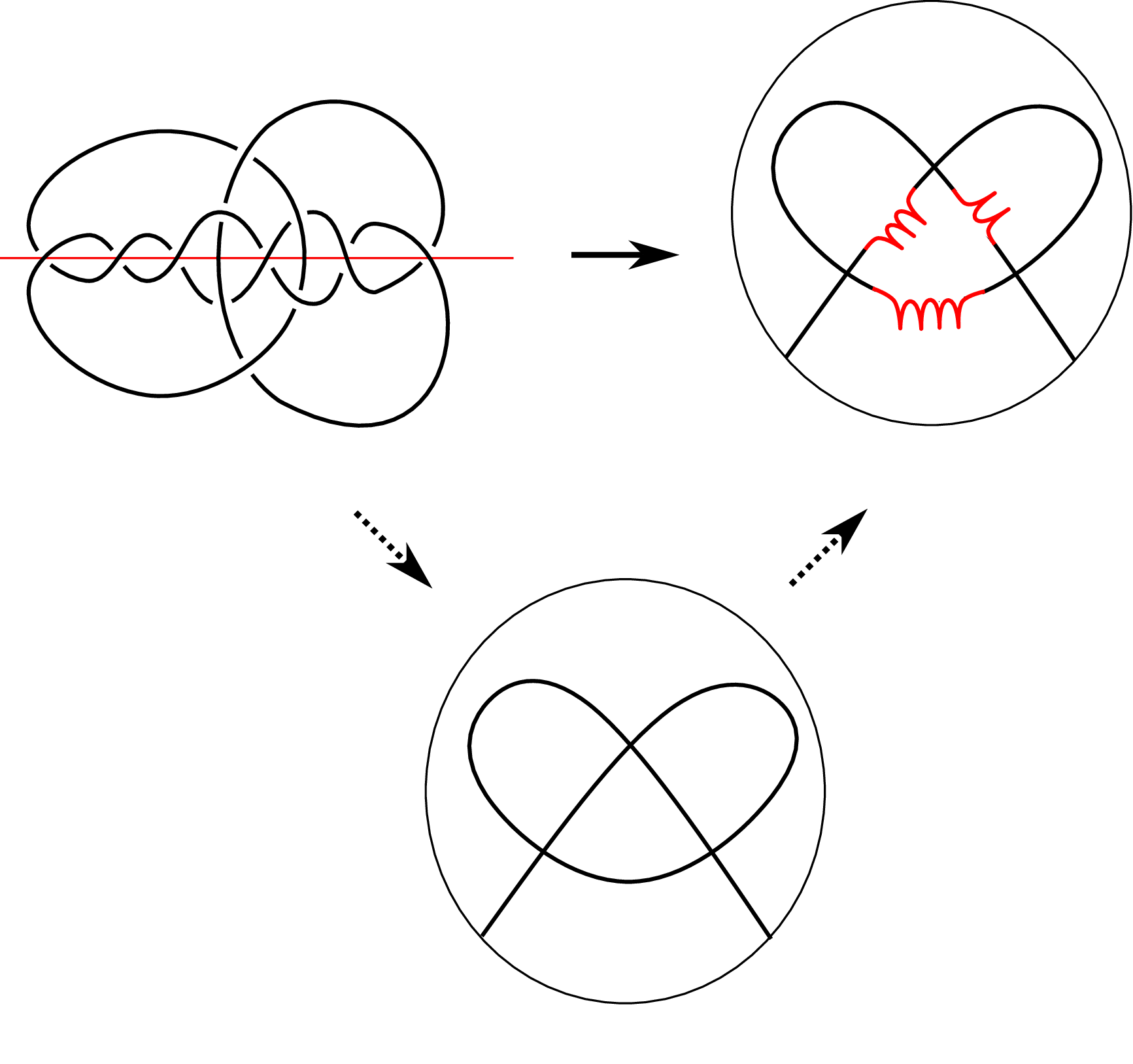}
\caption{The transvergent diagram $DL$ and the divide with cusps $P$. The divide with cusps $P$ is obtained from $DL$ via the divide with gleams $(P_{DL}, gl)$.}
\label{figure:example_main_theorem}

\begin{picture}(400,0)(0,0)
\put(184, 140){$-2$}
\put(240, 140){$3$}
\put(208, 120){$-\frac{5}{2}$}

\put(65, 188){Proposition \ref{prop:transvergent_diagrams_to_divides_with_gleams}}
\put(285, 188){Proposition \ref{prop:divides_with_gleams_to_divides_with_cusps}}

\put(100, 220){$DL$}
\put(296, 220){$P$}
\put(193, 60){$(P_{DL}, gl)$}

\end{picture}
\end{figure}
\end{proof}

\section{A transvergent diagram of the link of a divide with cusps}
\label{section:A transvergent diagram of the link of a divide with cusps}

It is known that there are algorithms for drawing the link of a given divide or a divide with cusps. Hirasawa \cite{Hira02} gave an algorithm to draw a diagram of the link of a given divide, and Chmutov \cite{C03} gave an algorithm to draw a transvergent diagram of the link of a divide. For a divide with cusps, Sugawara \cite{S24} gave an algorithm to draw a diagram of the link of a given divide with cusps. 
This Section provides an algorithm to draw a transvergent diagram of the link of a given divide with cusps.

\begin{theorem}
\label{theo:appendix}
Let $P$ be a divide with cusps.
The following steps produce a transvergent diagram $DL_P$ of the link $L_P$.

\begin{itemize}
\item[1.]
Let $(P',gl)$ be the divide with gleams obtained from $P$
by Proposition~\ref{prop:cusp_canonical_gleam}.
Here $P'$ is obtained from $P$ by removing all cusps.

\item[2.]
Take the double $P' \cup \widehat{P'} \subset D \cup \widehat{D}$ of the divide $P' \subset D$, where $\widehat{P'}$ and $\widehat{D}$ are the copies of the divide $P'$ and the $2$-disk $D$, respectively. Then, identify $D \cup \widehat{D}$ with the $2$-sphere $S^2$, and regard $P' \cup \widehat{P'}$ as a tetravalent graph on $S^2$.

\item[3.]
Assign arbitrary over/under information to each tetravaelnt vertex of $P' \cup \widehat{P'}$, thereby obtaining a transvergent diagram of a link in $S^{3}$. Denote this transvergent diagram by $DL_{0}$ and the link it represents by $L_{0}$.

\item[4.]
Let $gl_0$ be the gleam on $P'$ induced from $DL_0$
by Proposition~\ref{prop:transvergent_diagrams_to_divides_with_gleams}.

\item[5.]
For each internal region $R$ of $P'$, set
\[
n_R := gl(R)-gl_0(R) \in \ZZ.
\]
Let $\widetilde {l_R}$ be the unknot in $DL_0$ that intersects transversely each of the two regions corresponding to $R$ in exactly one point.

\item[6.] Obtain a transvergent  diagram $DL_{P}$ of $L_{P}$ by performing the $(-\frac{1}{2}n_{R})$-twists along $\widetilde{l}_{R}$ to $DL_{0}$ for all internal region $R$ of $P'$. \end{itemize}
\end{theorem}
\begin{example}
Let $P$ be the divide with cusps shown on the left in Figure \ref{figure:exapmle_appendix1}. Then, we obtain a transvergent diagram $DL_{P}$ of the link $L_{P}$ by following the steps of Theorem \ref{theo:appendix}, as described below.

\begin{itemize}
\item[1.] By applying Proposition \ref{prop:cusp_canonical_gleam} to the divide with cusps $P$, we obtain the divide with gleams $(P', gl)$ shown on the right in Figure \ref{figure:exapmle_appendix1}. 

\begin{figure}[htbp]
\centering\includegraphics[width=9cm]{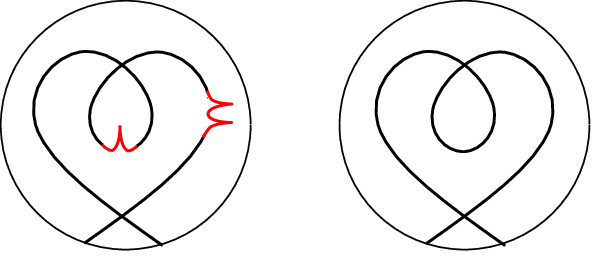}
\begin{picture}(400,0)(0,0)
\put(270, 80){$\frac{5}{2}$}
\put(262, 53){$-\frac{5}{2}$}
\put(195, 75){$=$}

\put(120, 0){$P$}
\put(257, 0){$(P', gl)$}
\end{picture}

\caption{The divide with cusps $P$ and its corresponding divide with gleams $(P', gl)$.}
\label{figure:exapmle_appendix1}
\end{figure}

\item[2.] We take the double of the divide $P'$ and obtain a transvergent diagram $DL_0$ by arbitrarily assigning over/under information to each tetravalent vertex of it. 
Then, we obtain a divide with gleams $(P_{DL_0}, gl_0)$ by applying Proposition \ref{prop:transvergent_diagrams_to_divides_with_gleams} to $DL_0$ (see Figure \ref{figure:exapmle_appendix2}).
\begin{figure}[htbp]
\centering\includegraphics[width=12cm]{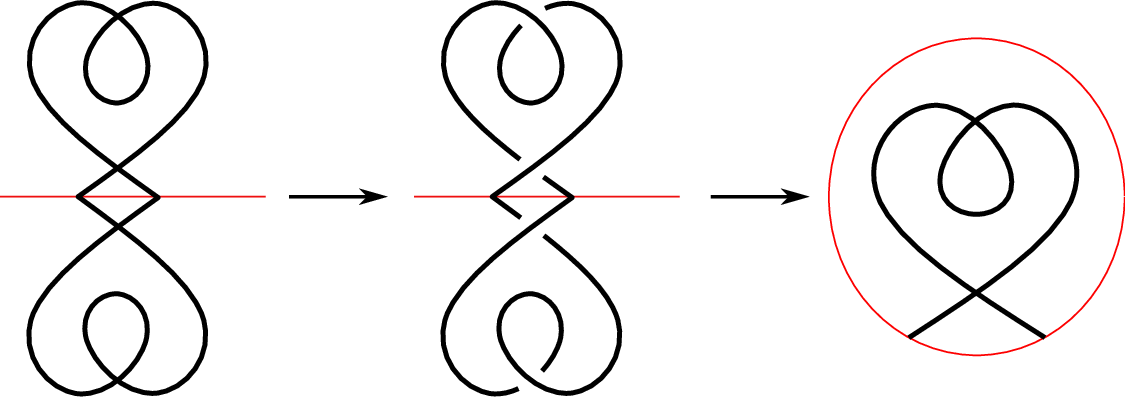}
\begin{picture}(400,0)(0,0)
\put(322, 84){$\frac{1}{2}$}
\put(315, 61){$-\frac{3}{2}$}

\put(298, 0){$(P_{DL_0}, gl_0)$}
\put(181, 0){$DL_{0}$}

\end{picture}
\caption{The double of the divide $P'$, the transvergent diagram $DL_0$, and the divide with gleams $(P_{DL_0}, gl_0)$.}
\label{figure:exapmle_appendix2}
\end{figure}

\item[3.] We obtain a transvergent diagram $DL_P$ of the link $L_{P}$ by performing $(-\frac{1}{2} n_{R_1})$-twists and $(-\frac{1}{2} n_{R_2})$-twists on $DL_0$ along the two unknots $\widetilde{l}_1$ and $\widetilde{l}_2$ in Figure \ref{figure:exapmle_appendix}, respectively. Here, $n_{R_1}:=gl(R_1)-gl_0(R_1)= -\frac{5}{2} -(-\frac{3}{2})=-1$, $n_{R_2}:=gl(R_2)-gl_0(R_2)= \frac{5}{2}-\frac{1}{2}=2$.
\begin{figure}[htbp]
\centering\includegraphics[width=13cm]{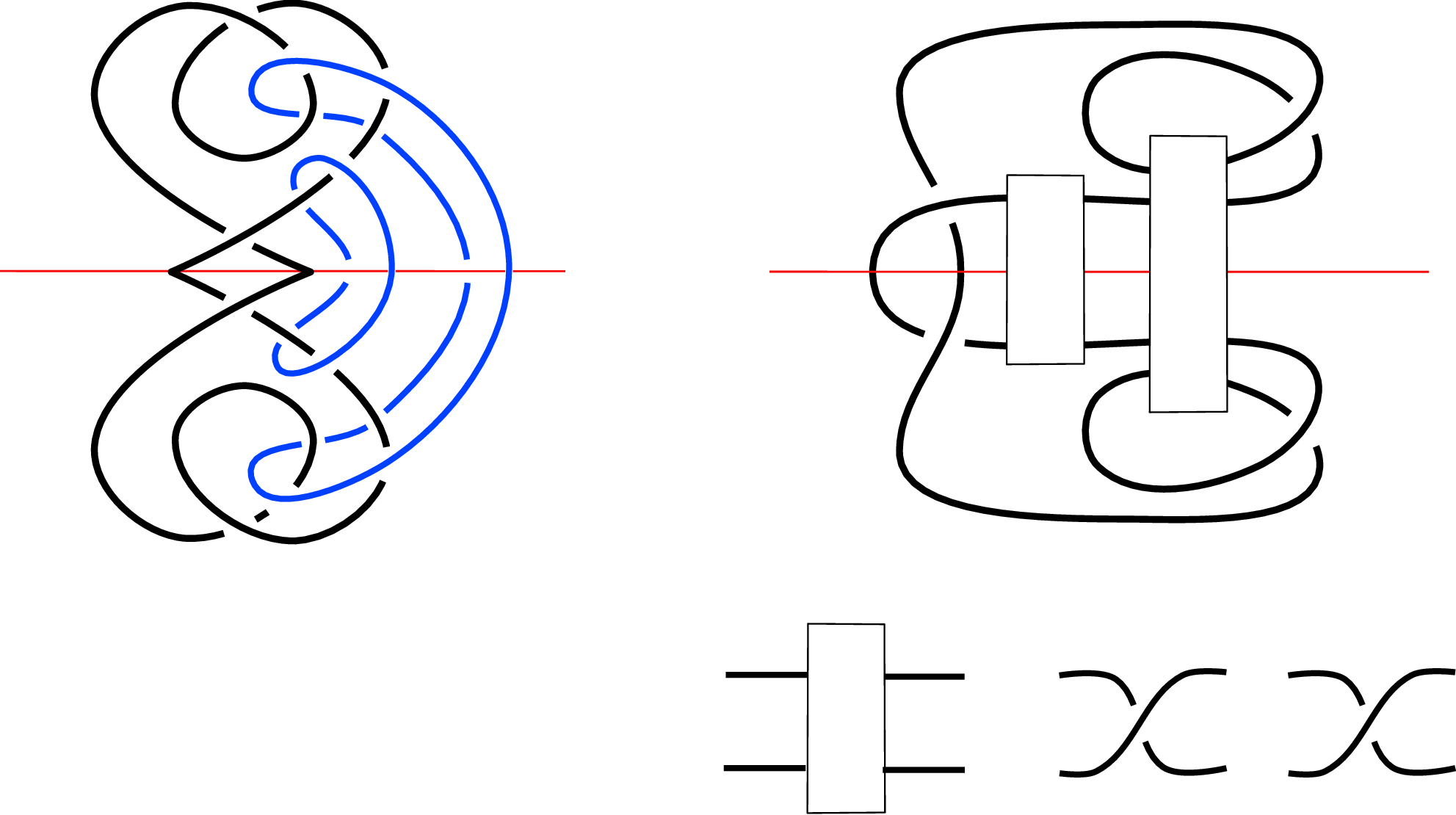}
\begin{picture}(400,0)(0,0)
\put(205, 33){$\vdots$}
\put(250, 33){$\vdots$}
\put(290, 33){$\vdots$}
\put(375, 33){$\vdots$}
\put(329, 35){$\cdots$}

\put(114, 129){$\widetilde{l}_{1}$}
\put(147, 129){$\widetilde{l}_{2}$}

\put(277, 148){$\frac{1}{2}$}
\put(310, 148){$-1$}

\put(226, 35){$\frac{n}{2}$}
\put(268, 33){$=$}
\put(305, 0){$\frac{n}{2}$-twists ($n>0$)}

\put(210, 200){$DL_P$}
\put(5, 200){$DL_0$}

\end{picture}
\caption{The transvergent diagram $DL_P$ of $L_P$ obtained from the diagram $DL_0$ of $L_0$ by performing $(-\frac{1}{2} n_{R_i})$-twists along $\widetilde{l}_i$ for $i=1, 2$.}
\label{figure:exapmle_appendix}
\end{figure}
\end{itemize}
\end{example}

\begin{proof}[Proof of Theorem \ref{theo:appendix}]
Let $q : S^3 \to S^3$ be the double covering branched along $\partial D$.
Since $L_0$ is a symmetric link whose symmetry axis is $\partial D$, 
the image $G_0:=q(L_0 \cup \partial D)$ is a trivalent graph in $S^3$.
Let $DG_0$ denote the diagram of $G_0$ induced from $DL_0$ by the branched covering map $q$. 
Let $Y_{DG_0}$ be the shadow of $(S^{3}, G_0)$ obtained by applying the method described in Example \ref{ex:shadows_of_DG} to the diagram $DG_0$. 
Let $R_1,\dots,R_n$ be the internal regions of $Y_{DG_0}$. 
For a rengion $R_i$, let $l_i$ is the preimage of a point in $R_i$ under the collapse $S^3 \to Y_{DG_0}$. 
Note that $l_i$ is an unknotted loop in $S^3$. 
By Turaev's reconstruction, 
the trivalent graph $G_P$ of the divide with cusps $P$ is obtained from $G_0$ 
by $\frac{1}{n_{R_i}}$--surgery along each $l_i$. 
Thus $G_P$ is obtained from $G_0$ by applying $(-n_{R_i})$--twists along each $l_i$.
These twists from $G_0$ to $G_P$ lift under the double branched covering $q$ and induce corresponding twists from $L_0 = q^{-1}(G_0)$ to $L_P = q^{-1}(G_P)$.
More precisely, since the double branched covering $q : S^3 \to S^3$ is expressed as \( q(x, re^{\theta i}) = (x, re^{2\theta i}) \), a twist along $l_i$ lifts to a half-twist along its lift $\widetilde{l}_i$. 
Consequently, the link $L_P$ of the divide with cusps $P$ is obtained from $L_0$ by applying $(-\tfrac12 n_{R_i})$--twists along each $\widetilde{l}_i$.
\end{proof}

\section*{Acknowledgments} 
The author would like to express his gratitude to Yuya Koda for his valuable comments and encouragement. 
The author also wishes to thank the anonymous referee for his or her valuable comments and suggestions that greatly improved the exposition.

\end{document}